\documentclass[12pt]{amsart}

\usepackage{geometry,graphicx,amssymb,amsmath,amsbsy,eucal,amsfonts,mathrsfs,amscd,bm,tcolorbox,mathtools,hyperref}

\usepackage[all]{xy}
\usepackage{tikz}
\usepackage{tikz-3dplot}
\usepackage{tkz-euclide}
\usetikzlibrary{cd}
\usepackage{pgfplots}
\usepackage{mathtools}
\usepackage{caption}
\usepackage{subcaption}
\usepackage{ulem}
\usepackage[bbgreekl]{mathbbol}

\numberwithin{equation}{section}

\allowdisplaybreaks[3]

\newcommand{\Krizek}{K\v{r}\'{\i}\v{z}ek}

\newcommand{\og}{\omega}
\DeclareMathOperator{\divg}{div}
\DeclareMathOperator{\Div}{Div}
\DeclareMathOperator{\dev}{dev}

\newtheorem{theorem}{Theorem}[section]
\newtheorem{lemma}[theorem]{Lemma}

\newtheorem{corollary}[theorem]{Corollary}
\newtheorem{proposition}[theorem]{Proposition}
\theoremstyle{definition}

\theoremstyle{remark}
\newtheorem{remark}[theorem]{Remark}

\tikzset{commutative diagrams/.cd,
mysymbol/.style={start anchor=center,end anchor=center,draw=none}
}

\newcommand{\dofcnt}[1]{{\tiny\text{{#1}}}}
\newcommand{\R}{\mathbb{R}}
\newcommand{\V}{\mathbb{V}}
\newcommand{\M}{\mathbb{M}}

\newcommand{\X}{\mathbb{X}}
\newcommand{\SSS}{\mathbb{S}}
\newcommand{\K}{\mathbb{K}}

\newcommand{\om}{{\Omega}}

\newcommand{\ds}{} % {\mathrm{ds}}

\renewcommand\div{\operatorname{div}}
\newcommand\tr{\operatorname{tr}}
\newcommand\sym{\operatorname{sym}}

\newcommand{\ust}{u^\star}

\newcommand{\uu}{\hat{u}}  % {\mathsf{u}}
\newcommand{\ssigma}{\hat{\sigma}}   % {\boldsymbol{\sigma}}
\newcommand{\tsigma}{\tilde{\sigma}}
\newcommand{\tu}{\tilde{u}}

\newcommand{\EK}{\mathsf{E}}
\newcommand{\dd}{\mathsf{d}}

\newcommand{\BDM}{{B\!D\!M}}
\newcommand{\BDMo}{\mathring{\BDM}}
\newcommand{\curl}{{\ensuremath\mathop{\mathrm{curl}\,}}}
\newcommand{\grad}{{\ensuremath\mathop{\mathrm{ grad}\,}}}
\newcommand{\dive}{{\ensuremath\mathop{\mathrm{div}\,}}} % space needed!

\newcommand{\pol}{\EuScript{P}}

\newcommand{\Aa}{\mathcal{A}}

\newcommand{\bn}{n}

\newcommand{\Ta}{T^{\text{A}}}
\newcommand{\Tha}{\mathcal{T}_h^{\text{A}}}
\newcommand{\Th}{\mathcal{T}_h}

\newcommand{\LL}{\mathcal{L}}
\newcommand{\LLo}{\mathring{\mathcal{L}}}

\newcommand{\SW}{\Sigma^w}
\newcommand{\VW}{V^w}
\newcommand{\PPi}{J} % {\mathbb{J}}  % (none of the other interpols are in mathbb)

\newcommand{\rank}{\operatorname{rank}}
\newcommand\skw{\operatorname{skw}}
\newcommand\vskw{\operatorname{vskw}}
\newcommand\mskw{\operatorname{mskw}}

\newcommand{\revj}[1]{{\color{black}#1}}

\newcommand{\RM}{\mathcal{R}}
\title[JKM elasticity element]{The Johnson--K\v{r}\'{i}\v{z}ek--Mercier Elasticity Element in Higher Dimensions}

\author[]{Jay Gopalakrishnan}
\address{
Portland  State University,
PO Box 751 (MTH), 
Portland, OR 97207, USA}
\email{gjay@pdx.edu}

\author[]{Johnny Guzm\'an}
\address{Division of Applied Mathematics,
Brown University,
Box F,
182 George Street,
Providence, RI 02912, USA}
\email{johnny\_guzman@brown.edu}

\author[]{Jeonghun J. Lee}
\address{Department of Mathematics,
Baylor University,
Sid Richardson Science Building,
1410 S.4th Street,
Waco, TX 76706, USA}
\email{jeonghun\_lee@baylor.edu}

\thanks{\revj{This work was supported in part by the National Science Foundation (NSF) under grant DMS-2409900 (Jay Gopalakrishnan), DMS-2309606 (Johnny Guzman), and DMS-2110781 (Jeonghun J. Lee).} The authors gratefully acknowledge the interaction opportunities provided by ICERM during their 2024 program.}


\begin{document}

\maketitle

\begin{abstract}
  Mixed methods for linear elasticity with strongly symmetric stresses of lowest order are studied in this paper. On each simplex, the stress space has piecewise linear components with respect to its Alfeld split (which connects the vertices to barycenter), generalizing the Johnson--Mercier two-dimensional element to higher dimensions. Further reductions in the stress space in the three-dimensional case (to 24 degrees of freedom per tetrahedron) are possible when the displacement space is reduced to local rigid displacements.  Proofs of  optimal error estimates of numerical solutions and improved error estimates via postprocessing and the duality argument are presented.   
\end{abstract}

%%%%%%%%%%%%%%%%%%%%%%%%%%%%%%%%%%%%%
%%%%%%%%%%%%%%%%%%%%%%%%%%%%%%%%%%%%%
\section{Introduction}

The classical Hellinger--Reissner variational principle of linear elasticity yields saddle point problems simultaneously seeking  the stress tensor field and the displacement vector field. A standard finite element discretization of the Hellinger--Reissner formulation then needs two finite element spaces, 
namely a stress space 
$\Sigma_h$ of symmetric matrix-valued piecewise polynomials whose row-wise divergence is square integrable,  
% % \subset H(\div)$ 
% for approximating the stress
% (on which divergence is applied row wise),
and a displacement space $V_h$ of vector fields with square-integrable components.  % \subset L^2$ for approximating the displacement.m
% The symmetry of the % However, the 
% stress tensor, a consequence of  conservation of angular momentum,.
% This competing requirements of symmetry constraint of $\Sigma_h$
Decades of research have illuminated the 
nontrivial difficulties in finding such a pair of spaces, $\Sigma_h \times V_h$, which also satisfy the Babu\v{s}ka--Brezzi stability conditions of saddle point problems.

To the best of our knowledge, the first such finite element space $\Sigma_h \times V_h$ with rigorous mathematical analysis was presented in \cite{johnson1978some} for the two-dimensional elastostatic problem (and the element itself appeared earlier in~\cite{Watwood1968AnES}). This  finite element on triangular meshes, often called the Johnson--Mercier element, uses a Clough--Tocher  refinement~\cite{clough1965finite,ciarlet2009intrinsic} (also known as
HCT refinement~\cite{Ciarl91})
of a given mesh to define the shape functions of the discrete stress space $\Sigma_h,$ whereas the discrete displacement space $V_h$ is composed of piecewise linear functions defined on the original mesh. More precisely, $\Sigma_h$ is the space of symmetric tensors such that divergence of each of its rows is square-integrable, and each of its components is piecewise linear on the Clough--Tocher refinement. A family of higher order elasticity elements in two dimensions with the Clough--Tocher refinement of triangular meshes was developed in~\cite{arnold1984family}, where the  displacement spaces $V_h$ also have piecewise polynomials on the Clough--Tocher refinement. It is interesting to note that these same composite triangulations have also been critical to develop stable Stokes pairs that are pointwise divergence-free (e.g. \cite{ArnolQin92,Zhang-barycentric:2005}). More recently, the work of \cite{Christiansen-Hu:2018} led to renewed interest as they showed how these elements fit within discrete de Rham complexes and later in the elasticity complex \cite{christiansen2019finite}.

The goal of this work is an investigation of a higher dimensional
version of Johnson--Mercier element on Alfeld splits, the natural
generalization of Clough--Tocher splits to tetrahedra.  At the outset, we
were unaware that an extension of the Johnson--Mercier element to 
three dimensions (3D) was already done in 1982 by \Krizek\ in
\cite{Krizek1982}. This is even before Alfeld's work~\cite{Alfel84}
after which the tetrahedral split was known as the Alfeld split.
% on tetrahedral meshes with a generalization of the Clough--Tocher refinement to the three-dimensions, which will be called Alfeld split in this paper \cite{Alfel84}. 
A cursory literature search gives the impression that \Krizek's work
is largely unknown at present: it is not mentioned in influential
previous studies
like~\cite{arnold2008finite,hu2015family,hu2016low,gong2023discrete,Stenberg-2023,Huang-2023lowordermixed}.
Nonetheless, in~\cite{Krizek1982}, a proof of unisolvency of the 3D version of
the Johnson--Mercier stress element
and error estimates for stress are
presented.
In this paper, we present a different
proof of unisolvency of the same 3D stress element, 
which we will refer to as the Johnson-\Krizek-Mercier (JKM) element from now on.
We show that our
new technique of proof generalizes to give a stress element in higher
dimensions. Even in 3D, the new analysis helps us develop a
reduced stress element with only 24 degrees of freedom per element,
which to our knowledge, is the simplest conforming symmetric stress
element known currently.  We also prove error estimates for the
displacement, prove superconvergence estimates by duality argument,
show robustness in the incompressible limit, and develop a
postprocessing to enhance convergence rates. We also study the method
obtained by pairing the same stress space with piecewise constant
displacements with respect to the Alfeld refinement (instead of the
piecewise linear displacement spaces that Johnson and Mercier used).
All these elements are illustrated and summarized in Figure~\ref{fig:figdofs}.

Our approach is based on the analysis
of~\cite{arnold1984family,johnson1978some}, rather than that of
\cite{Krizek1982}. The unisolvency proofs
in~\cite{arnold1984family,johnson1978some} of stress finite elements
$\Sigma_h$ proceed by characterizing the divergence-free subspace of
$\Sigma_h$ that have vanishing normal components on the edges of a
triangular element. They show that this subspace must consist of
the Airy stress tensor of $C^1$ splines on the Clough--Tocher split. In the more modern
language of finite element theory, this fits in an exact local discrete
complex of the corresponding spaces with homogeneous boundary conditions.
Note
that $\dive \Sigma_h$ is not a subspace of the displacement space
$V_h$ used in the JKM element, so one would have to
replace $V_h$ with piecewise constant displacements on the
Clough--Tocher triangulation to relate it to the complex. In contrast,
the 3D analysis in \cite{Krizek1982}
utilizes special geometric properties of tetrahedra and inter-element
continuity of polynomials for unisolvency proof, an argument
we could not  generalize to higher dimensions. For our extension of 
Johnson--Mercier element to higher  dimensions, the key, borrowed from
developments in  finite element exterior calculus, is to exploit an
identity used in the BGG resolution~\cite{arnold2007mixed,calabi1961compact, eastwood2000complex}. We give an elementary proof of
this identity in arbitrary dimensions.  Although we have chosen to
present this generalization without using the language of exterior calculus, 
it is  motivated by the considerations to make the local stress
space fit into an exact discrete elasticity complex. Indeed, similar
ideas were critical in \cite{christiansen2020discrete} where the first
3D discrete elasticity complex was developed on Alfeld splits. Since
we only need to focus on the last part of such an $N$-dimensional
complex to obtain the generalization of the Johnson--Mercier element, a
complete study of the full $N$-dimensional complex is outside the
scope of this paper.

% was developed in three dimensions. 
% In this paper we extend the discrete elasticity complex \cite{christiansen2020discrete} to any dimensions and it leads to construction of Johnson--Mercier element in any dimensions.

Composite conforming stress elements are interesting because they promise
liberation from the necessity of vertex degrees of freedom when using
polynomial spaces on unsplit elements~\cite{arnold2008finite}. Due to complications with supersmoothness, this promise was not realized on Alfeld splits 
in~\cite{christiansen2020discrete}, but was realized on the more involved  Worsey-Farin splits (which splits a tetrahedron to 12 subtetrahedra) later in \cite{gong2023discrete}. Through a new analysis of the lowest order case, in this paper we are finally able to provide a 3D stress element on Alfeld splits without vertex and edge degrees of freedom. Even in higher dimensions, our composite stress element does not have degrees of freedom on subsimplices of codimension greater than one. In particular, this allows for a natural facet-based hybridization. % There are two previous works on elasticity finite elements with such properties in the three dimensions \cite{Krizek1982,gong2023discrete}. The elements in \cite{gong2023discrete} use the Worsey-Farin split which splits one tetrahedron of an original mesh to 12 subtetrahedra, so the construction of finite element is more involved. In this paper, we construct such elements in any dimensions on general Alfeld splits on simplicial meshes. 
Another nice property of our elasticity element is that only low-degree piecewise polynomials are used for the stress finite element. Most symmetric conforming stress  elements for 3D elasticity contain piecewise quadratic or higher degree polynomials for stress (although \cite{gong2023discrete} is an exception). Only some reduced elements have the stress shape functions which do not contain all quadratic polynomials~\cite{arnold2002mixed}, but contains other high polynomial degree shape functions. In this paper, even in  $N$-dimensions,  only piecewise linear polynomials are used for the stress. Finally, we point out that an efficient implementation of the (non-reduced) JKM element appears to be already underway~\cite{Brubeck-Kirby:2025,2025-defelement}.
% Moreover, we propose a reduced mixed finite element such that the stress elem ent does not contain full piecewise linear polynomials and the displacement element contains local rigid motions. The reduced element has 24 local degrees of freedom for the stress and 6 local degrees of freedom for the displacement in the three dimensional problems. To the best of our knowledge, this is the simplest dual mixed finite element for the Hellinger--Reissner formulation of the three dimensional elasticity with $H(\div)$-conforming symmetric stress finite element. 

The paper is organized as follows.  We begin by defining the elasticity problem in Section~\ref{sec:JM-element}. There we present our unisolvency proof of the JKM element. We then analyze the stability and the error estimates of numerical solutions with the JKM element. Optimal error estimates of numerical solutions and superconvergent error estimates of post-processed solutions are proved. In Section~\ref{sec:P0-disp}, another numerical method with the JKM stress element paired with piecewise constant displacement, is proposed and analyzed. We then present  a reduced finite element pair in Section~\ref{sec:reduced-element} and prove its error estimates. The connection between  the JKM element and a linear elasticity finite element method with weakly symmetric stress is discussed in Section~\ref{sec:weak-symmetry}. Generalization of the Johnson--Mercier element to $N$-dimensions ($N\ge 2$) is given in Section~\ref{sec:ndim-case}. Finally, a further reduced finite element pair and its error estimates are discussed in Appendix~\ref{sec:2nd-reduced-element}.

%The extension to $N$-dimensions utilizes the elasticity complex with weakly symmetric stress discovered in \cite{arnold2007mixed}. For self-contained presentation without the exterior calculus, proxy tensor field and their commuting diagram properties are presented. %~\ref{sec:proxy-tensor-fields} and \ref{sec:commuting-tensor-fields}.

\section{The JKM Element for
  3D Elasticity}
\label{sec:JM-element}

\subsection{Notation}  \label{ssec:notation}
%Here we summarize notation and definitions. 
For a measurable set $D \subset \mathbb{R}^N$ $(N \ge 2)$, let $L^2(D)$ denote the Lebesgue space with $N$-dimensional Lebesgue measure. For a domain $D \subset \mathbb{R}^N$ with Lipschitz boundary, $H^k(D)$, $k \ge 1$, denotes the standard Sobolev space \cite{Evans-book} of functions in  $L^2(D)$ all of whose derivatives
of up to $k$th order are also in $L^2(D)$.
Let $\M$ denote the space of $N\times N$
matrices. For $M \in \M$, let 
$  \text{sym } M= \frac{1}{2}(M+M'),$
  $  \text{skw } M= \frac{1}{2}(M-M'),$
and  let $\SSS = \sym \M$ and  $\K = \skw(\M)$.
For   vector spaces $\mathbb{X}$ such as  $ \R, \R^N, \M, \SSS, \K,$ etc.,
we let $L^2(\Omega, \mathbb{X})$ denote the space of $\mathbb{X}$-valued functions whose components are in $L^2(\Omega)$. The inner product of $L^2(D, \mathbb{X})$ and the corresponding norm are denoted $(\cdot, \cdot)_D$ and $\| \cdot \|_{L^2(D, \X)}$, respectively.
Similarly, the $\mathbb{X}$-valued Sobolev spaces and norms are $H^k(D, \mathbb{X})$ and $\| \cdot \|_{H^k(D,\mathbb{X})}$. The $\X$ in the subscripts of  (semi)norm notation will be omitted  if there is no concern of confusion.
Furthermore, if $D$ is  the domain $\Omega$ of  our boundary value problem, then we abbreviate  $(\cdot, \cdot)_\om$ to $(\cdot, \cdot)$. On matrix-valued functions, the divergence, denoted by $\div$, is calculated row wise. Doing so in the sense of distributions, we define 
\begin{align*}
    H(\div,D,\M):=\{ \omega \in L^2(D, \M): \div \omega \in L^2(D, \R^N) \}
\end{align*}
For a set $D$ in $d$-dimensional hyperspace ($d \le N$) and a nonnegative integer $k$, $\mathcal{P}_k(D)$ is the space of polynomials of degree at most $k$, restricted to the domain $D$; similarly, $\mathcal{P}_k(D,\mathbb{X})$ denotes the space of $\mathbb{X}$-valued polynomials of degree $\le k$. 

\subsection{Elasticity equations}
Let $\Omega \subset \mathbb{R}^3$ be a domain with a Lipschitz polyhedral boundary. We denote the spaces of three-dimensional column vectors and the space of $3\times 3$ symmetric matrices by $\mathbb{R}^3$ and $\mathbb{S}$. We consider the linear elasticity equations with homogeneous Dirichlet boundary conditions:
\begin{subequations} \label{strong-eqs}
\begin{alignat}{2}
\Aa \sigma=& \epsilon (u) \qquad && \text{ in } \Omega, \\
\dive \sigma=& f \qquad  && \text{ in } \Omega, \\
u=& 0 \qquad && \text{ on }  \partial \Omega
\end{alignat}
\end{subequations}
where $u:\Omega \rightarrow \mathbb{R}^3$ is the displacement vector field, $\sigma:\Omega \rightarrow \mathbb{S}$ is the stress tensor field,  $f:\Omega \rightarrow \mathbb{R}^3$ is the load vector, and $\epsilon(u)=\sym (\nabla u)$
is the symmetrized gradient. Here $\Aa$ is the compliance tensor, a rank-4 tensor field with entries $\Aa_{ijkl}:\Omega \rightarrow \mathbb{R}$, $1\le i, j, k, l \le 3$ which satisfy symmetry conditions 
\begin{align} \label{A-symmetry}
\Aa_{ijkl} = \Aa_{jikl}=\Aa_{ijlk}=\Aa_{klij}. 
\end{align}
For  $\tau \in \mathbb{S}$ with entries $\tau_{ij}$, $1\le i,j\le 3$, $\Aa(x) \tau$ for $x \in \Omega$ is the matrix whose $(i,j)$-entry is defined by $\sum_{1 \le k, l \le 3} \Aa_{ijkl}(x) \tau_{kl}$. From \eqref{A-symmetry}, one can check that $\Aa(x) \tau$ also belongs to $\mathbb{S}$, and $\Aa(x)$ gives a bilinear form on $\mathbb{S}$ by $\Aa(x) \tau: \omega$ for $\tau, \omega \in \mathbb{S}$ where the colon stands for the Frobenius inner product of matrices. We assume that $\Aa$ is positive definite over symmetric matrices and is bounded, namely
there are positive constants $\theta, \gamma>0$ such that
\begin{alignat}{1}
\Aa(x) \omega : \omega & \ge  \theta \omega: \omega, \quad \omega \in \mathbb{S}, x \in \Omega, \label{A-coercive}\\   
\|\Aa\|_{L^\infty(\Omega)}   &  \le \gamma. \label{A-bounded}
\end{alignat}
In our results for nearly incompressible elastic materials, we make an alternate assumption since \eqref{A-coercive} does not generally hold in such cases: instead of~\eqref{A-coercive}, we then assume that
there exists a  $\theta>0$ such that 
\begin{align} 
\Aa(x) \omega : \omega  & \ge \theta \,\text{dev}(\omega): \text{dev}(\omega), \quad \omega \in \mathbb{S}, x \in \Omega \label{A-incompressible}
\end{align}
where $\text{dev}(\omega) := \omega - \revj{\frac 13} \tr(\omega)\mathbb{I}$ with the $3 \times 3$ identity matrix $\mathbb{I}$. %, called the deviatoric part of $\omega$.
For isotropic materials, as the Lame material parameter $\lambda \to \infty$, we can expect \eqref{A-incompressible} to hold, but not \eqref{A-coercive}.

A well-known variational formulation of \eqref{strong-eqs} is as follows: Find $\sigma \in \Sigma$,  $u \in V $ such that
\begin{subequations}\label{Elas}
\begin{alignat}{2}
(\Aa \sigma, \tau)+(u, \dive \tau)=& 0 \qquad && {\text{ for all }}  \tau \in \Sigma, \label{Elas1}\\
(\dive \sigma, v)=& (f,v) \qquad  && {\text{ for all }} v \in V \label{Elas2}
\end{alignat}
\end{subequations}
where $\Sigma=H(\div, \Omega, \mathbb{S}):= \{ \omega \in L^2(D, \mathbb S): \div \omega \in L^2(D, \R^3) \}$, $V:= [L^2(\Omega)]^3$. 

\subsection{The finite element}  \label{ssec:finite-element}
A canonical finite element discretization of \eqref{Elas} is to find $\sigma_h \in \Sigma_h$,  $u_h \in V_h $ such that
\begin{subequations}\label{Elas-disc}
\begin{alignat}{2}
(\Aa \sigma_h, \tau)+(u_h, \dive \tau)=& 0 \qquad && {\text{ for all }} \tau \in \Sigma_h , \label{Elas-disc1}\\
(\dive \sigma_h, v)=& (f,v) \qquad  && {\text{ for all }} v \in V_h  \label{Elas-disc2}
\end{alignat}
\end{subequations}
with suitable finite element spaces $\Sigma_h \subset \Sigma$, $V_h \subset V$. It is well-known that the finite element pair $\Sigma_h, V_h$ needs to satisfy the Babu\v{s}ka--Brezzi stability conditions for existence and uniqueness of numerical solutions and for accurate approximation of exact solutions (see, e.g., \cite{BoffiBrezziFortin-2013}).

To construct a finite element subspace $\Sigma_h$, 
we  work on Alfeld simplicial complexes and start by establishing notation associated to an Alfeld split.
Starting with a tetrahedron $T=[x_0, x_1, x_2,  x_3]$,
let $\Ta$ be an Alfeld triangulation of $T$, i.e.,
we choose an interior point $z$ of $T$ and
we let $T_0=[z, x_1, x_2, x_3]$,  $T_1=[z, x_0, x_2, x_3],$
$T_2=[z, x_0, x_1, x_3],$
$T_3=[z, x_0, x_1, x_2]$
and set $\Ta= \{ T_0, T_1, T_2, T_3\} $.

We let $\Th$ be a conforming triangulation $\Omega$ with tetrahedra and $\Tha$ be the resulting mesh after performing an Alfeld split on each $T \in \Th$. 
For $T \in \Th$ we define the  local spaces as 
\begin{alignat*}{1}
    \Sigma_h(T):=&\{ \omega \in H(\dive, T, \mathbb{S}): \omega|_K \in \pol_1(K, \mathbb{S}), \text{ for all }  K \in \Ta\}, \\
    V_h(T):=&  [\pol_1(T)]^3.
\end{alignat*}
Although  $\divg( \Sigma_h(T))$ is not contained  in $V_h(T)$,
it is contained in 
\begin{alignat*}{1}
W_h(T):=&\{ v \in [L^2(T)]^3: v \in [\pol_0(K)]^3, \text{ for all }  K \in \Ta\}.
\end{alignat*}
Note that $W_h(T)$ and  $V_h(T)$ have the same dimension.  We let 
 $P_T$ denote the $L^2$-orthogonal projection onto $W_h(T)$:
 \begin{alignat}{1}
   \label{eq:PTdef}
  (P_T v, w)_T=(v, w)_T \qquad {\text{ for all }} w \in W_h(T), \; {\text{ for all }} 
  T \in \Th.
\end{alignat}
For every $K \in \Ta$ let $x_K$ be the barycenter of $K$. We let $I_T: W_h(T) \mapsto V_h(T)$ be  defined by  
\begin{equation}\label{IT}
I_T w(x_K)= w(x_K) \qquad \text{ for } K \in \Ta.
\end{equation}
A simple quadrature argument gives that
\begin{equation}\label{ITinner}
(I_T w, m)_T=(w, m)_T \qquad {\text{ for all }}  w,m \in W_h(T),
\end{equation}
which, in particular, shows that $I_T$ is an injection.  
Since $\dim W_h(T) = \dim V_h(T),$
injectivity of $I_T$  implies bijectivity of  $I_T$. 
Moreover, it is easy to see from~\eqref{eq:PTdef} and~\eqref{ITinner} that
\begin{equation}
  \label{eq:PhIh}
  P_T I_T w = w, \qquad \text{ for all } w \in W_h(T),
\end{equation}
so the restriction of $P_T$ to $V_h(T)$, namely $P_T|_{V_h(T)}: V_h(T) \to W_h(T),$ and  $I_T: W_h(T) \to V_h(T)$  are inverses of each other.

The next lemma is crucial for identifying degrees of freedom of the stress element  $\Sigma_h(T)$. First we recall some preliminaries that are standard (see e.g.\ \cite{arnold2008finite}). Let $\M$ denote the space of $3\times 3$
matrices and let $\Xi : \M \to \M$ be defined by 
\begin{equation}
  \label{eq:Xi-defn}
  \Xi M=M' -\text{tr}(M) \mathbb{I},
\end{equation}
where $(\cdot)'$ denotes the transpose. The inverse of this operator is given by 
$\Xi^{-1} M=M'-\frac{1}{2}\text{tr}(M) \mathbb{I}.
$
We define  $\mskw : \mathbb{R}^3 \to \K$ and $\text{vskw}$ by
\begin{align*}
 \text{mskw} \begin{pmatrix}
v_1 \\
v_2 \\
v_3
\end{pmatrix}
 :=
  \begin{pmatrix}
0  &  -v_3 & v_2 \\
v_3 & 0 & -v_1 \\
-v_2& v_1 & 0 
\end{pmatrix}, \qquad \text{vskw}:= \text{mskw}^{-1} \circ \text{skw }.
\end{align*}
\revj{ More explicitly, if $M=(m_{ij})$, $1\le i,j\le 3$, 
\begin{align*}
\text{vskw} M= \frac{1}{2} \begin{pmatrix}
m_{32}-m_{23} \\
 m_{13}-m_{31}\\
m_{21}-m_{12}
\end{pmatrix}.
\end{align*}
} 
It is easy to see that 
\begin{gather}
\label{alg1-1}
  \dive \Xi= 2\text{vskw } \curl
  % \\
  % \label{alg2}
  % \Xi \grad=-\curl \text{ mskw}.
\end{gather}
holds \revj{where $\curl$ is the row-wise $\curl$ for $\M$-valued functions}.
Let $\pol_p(\Ta) = \{ u: u|_K \in \pol_p(K)$ for all
$K \in \Ta\}$ and let $\LLo_p(\Ta) = \pol_p(\Ta) \cap \mathring{H}^1(T)$ \revj{where $\mathring{H}^1(T)$ is the subspace of functions in $H^1(T)$ with vanishing traces on $\partial T$}.
Using this notation, we prove the next lemma showing that the space
of divergence-free elements with vanishing traces of $\Sigma_h(T)$ is
trivial. Our proof of this lemma  marks the main point of departure from the analysis of~\cite{Krizek1982}. An alternative to  the approach of~\cite{Krizek1982}, 
is to prove this lemma using local discrete elasticity sequences
on Alfeld splits obtained by small modifications of the recent 
three-dimensional exact sequences
in~\cite{christiansen2020discrete}. Nonetheless,  we have chosen to present
a new proof that is easier and generalizable to higher dimensions
(as we shall  demonstrate in Section~\ref{sec:ndim-case}).
Throughout,  $n$ denotes the unit outward normal vector field on the boundary of a  domain under consideration.

\begin{lemma}\label{Sigmazero}
  The equality 
  $\{ \omega \in \Sigma_h(T): \dive \omega=0, \omega n|_{\partial T}=0\}=\{0\}$
  holds. 
%The dimension of  the space $\Sigma_h(T)$ is $42$. 
\end{lemma}
\begin{proof}
  Consider the matrix analogue of the well-known lowest order Brezzi-Douglas-Marini (BDM) space, namely 
  \begin{align*}
    \BDM_1(\Ta)= \{ \omega \in H(\div, T, \M): \omega|_K \in \pol_1(K, \M) \text{ for all } K \in \Ta\}  
  \end{align*}
  and let $\BDMo_1(\Ta)=\{ \omega \in \BDM_1(\Ta): \omega n=0 \text{ on } \partial T \}$.
Note that 
\begin{align*}
    \{ \omega \in \Sigma_h(T): \dive \omega=0, \omega n|_{\partial T}=0\}=\{ \omega \in \BDMo_1(\Ta) : \; \vskw \omega =0, \; \dive \omega=0\}. 
\end{align*}
Let $\tilde \pol_0(\Ta) = \{ u \in \pol_0(\Ta): \int_T u = 0\}$ and let
  \[
    X = \left\{ (u, v):  u \in [\pol_1(\Ta)]^3, \; v \in [\tilde \pol_0(\Ta)]^3, \;
      \int_T (u \cdot c + v \cdot (c \times x)) = 0 \text{ for all } c \in \R^3
    \right\}.
  \]
  The proof is based on  the operator $A : \BDMo_1(\Ta) \to X$ given by
  \begin{equation}
    \label{eq:Adefn}
    A \eta = (\vskw \eta, \dive \eta), \qquad \eta \in \BDMo_1 \revj{(\Ta)}.
  \end{equation}
  Observe that for any $c \in \R^3$ and any $\eta \in \BDMo_1(\Ta)$, 
  \begin{equation}
    \label{eq:5}
    \int_T c \cdot \vskw \eta = \int_T \eta : \mskw c
    = \int_T \eta : \grad ( c \times x)
    = -\int_T \div \eta \cdot (c \times x).    
  \end{equation}
  Thus $A$ indeed maps into $X$. We proceed to  show that $A$ is surjective.

  Let $(u, v) \in X$. Since $v$ has components of zero
  mean, by a  standard exact sequence property (e.g. \cite{arnold2006finite}), % ~\cite{fu2018exact},
  there exists a $\sigma \in \BDMo_1(\Ta)$ such that
  \begin{equation}
    \label{eq:divsig}
    \dive \sigma = v.
  \end{equation}
  Since $(u, v ) \in X$, this implies 
  \[
    0 =
    \int_T (u \cdot c + v \cdot (c \times x))
    = \int_T (u \cdot c   + \div \sigma \cdot (c \times x))
    =
    \int_T (u - \vskw \sigma) \cdot c    
  \]
  for all $c \in \R^3$, 
  where we have again used~\eqref{eq:5} to obtain the last equality.
  Hence by a well-known exact sequence property~\cite{zhang2005stokes, guzman2018inf, fu2018exact} on
  Alfeld splits, there exists an $\eta \in \LLo_2(\Ta, \M)$ such that
  \begin{equation}
    \label{eq:diveta}
    \dive \eta = 2(u - \vskw \sigma).
  \end{equation}
  Now $\tau = \sigma + \curl \Xi^{-1} \eta$ is obviously in $\BDMo_1(\Ta)$
  and satisfies
  \begin{align*}
    \vskw \tau & = \vskw \sigma + \frac 1 2 \dive \eta = u, 
    && \text{(by~\eqref{alg1-1} and~\eqref{eq:diveta})}
    \\
    \dive \tau & = v. && \text{(by~\eqref{eq:divsig})}
  \end{align*}
  Thus we have proved that $A$ is surjective.

  To conclude, the surjectivity of $A$ implies that
  $\rank(A) = \dim X \ge 3 \dim \pol_1(\Ta) + 3\dim \tilde \pol_0(\Ta) - 3 = 54$. Since $\dim \BDMo_1 (\Ta) = 3 \cdot 6 \cdot 3=54$ on an Alfeld
  split (see e.g., \cite[eq.~(2.3)] {christiansen2020discrete}), 
  the rank-nullity theorem implies that the null space of $A$ is trivial. It means that 
  $ \{ \omega \in \BDMo_1 : \; \vskw \omega =0, \; \dive \omega=0\} = \{ 0\}$, so the conclusion follows. 
\end{proof}

\begin{figure}
  \centering
  \includegraphics[width=0.75\textwidth]{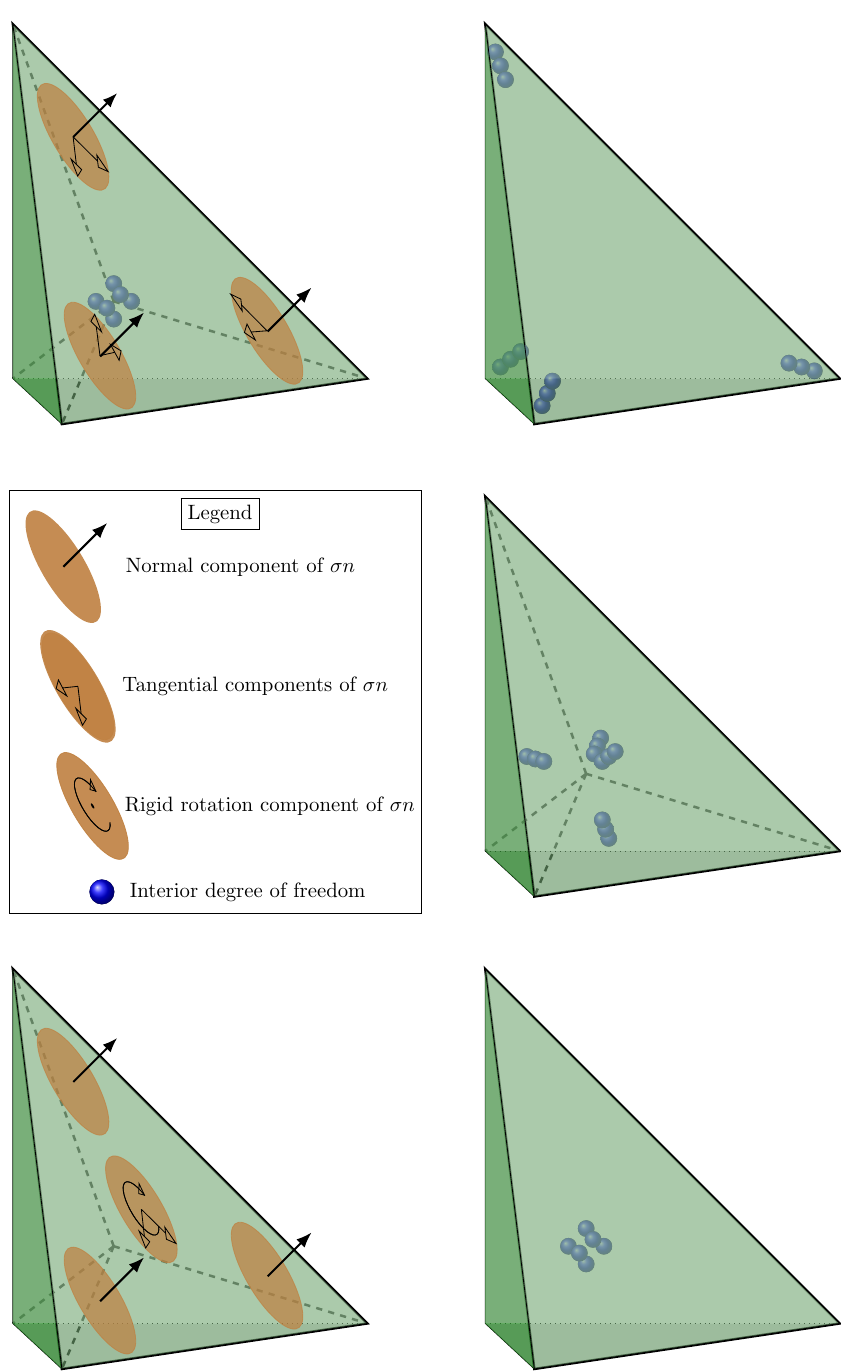}
  \caption{Illustrations of element degrees of freedom, with coupling degrees of freedom shown only on a single foreground face. {\em Top row:} Stress element ($\Sigma_h$) on the split element and displacement element ($V_h$) on the unsplit element for the method of Section~\ref{sec:JM-element}. {\em Middle right:} Split-element  displacements ($W_h$) of Section~\ref{sec:P0-disp} to couple with the same $\Sigma_h$. {\em Bottom row:} The reduced pair of
    Section~\ref{sec:reduced-element}.}
  \label{fig:figdofs}
\end{figure}

We can now state unisolvent degrees of freedom (``dofs'') for the
space $\Sigma_h(T)$. Let $\triangle_k(T)$ denote the set of $k$-subsimplices of a simplex $T$, and let $\triangle_k(\mathcal{T}_h)$ be the union of $\triangle_k(T)$ for all $T \in \mathcal{T}_h$ \revj{(see \cite{arnold2006finite})}. % In the discussions below, $n$ in an integral on face will denote a unit normal vector on the face. 
\begin{theorem}\label{thmdofs1}
The dimension of $\Sigma_h(T)$ is 42. Moreover, an element $\omega \in  \Sigma_h(T)$  is uniquely determined by the following dofs:
\begin{subequations}
\label{Sigmah}
\begin{alignat}{4}
    &\int_F \omega \bn \cdot \kappa , \qquad
  &&\kappa\in [\pol_1(F)]^3,
     \;F \in \triangle_2(T), \qquad && \dofcnt{(36 dofs)}\label{Sigmah1}
     \\ 
    &\int_T \omega. &&\quad && \dofcnt{($6$ dofs)}\label{Sigmah2}
\end{alignat}
\end{subequations}
\end{theorem}
\begin{proof}
Since $\Sigma_h(T)=\{ \omega \in \BDM_1(\Ta): \int_T \omega : \eta = 0 \; \text{ for all } \eta \in \pol_1(\Ta, \mathbb{K})\}$, 
  \begin{align*}
    \dim \Sigma_h(T) &\ge \dim \BDM_1(\Ta) - \dim \pol_1(\Ta, \mathbb{K})
    \\
    & = 90-48=42,
  \end{align*}
which is exactly the number of dofs given in \eqref{Sigmah}. We now show that $\dim \Sigma_h(T) \le 42$.

Let $\omega \in  \Sigma_h(T)$  for which  all dofs of \eqref{Sigmah} vanish. We will  show that $\omega=0$. The dofs \eqref{Sigmah1} give that $\omega \bn=0$ on $\partial T$. Let $v \in \pol_1(T)$. Then integration by parts gives
\begin{equation}
  \label{eq:1-div-om}
\int_{T} \divg \omega   \cdot v = -\int_{T} \omega : \epsilon(v) =0
\end{equation}
where we used \eqref{Sigmah2} in the last equality. Note that $\divg \omega \in W_h(T)$. Hence, substituting  $v=I_T \divg \omega$ in~\eqref{eq:1-div-om} and using~\eqref{ITinner} we get  $0= \int_{T} v \, \divg \omega =\int_{T} |\divg \omega|^2$ which shows that $\divg \omega =0$.
By Lemma~\ref{Sigmazero}, $\omega$ vanishes.
This proves that $\dim \Sigma_h(T) \le 42$, so $\dim \Sigma_h(T) = 42.$ The argument also shows that the dofs~\eqref{Sigmah} uniquely determine an element $\omega \in \Sigma_h(T)$.
\end{proof}

Unisolvent dofs of a finite element generate an associated canonical
interpolant.  Let $\Pi_T$ denote this canonical projection into
$\Sigma_h(T)$. It satisfies
\begin{subequations}
\label{Pi}
\begin{alignat}{4}
    &\int_F \Pi_T \omega \bn \,\cdot  \kappa = \int_F  \omega \bn \cdot \kappa , \qquad &&\kappa\in [\pol_1(F)]^3, \; F \in \triangle_2(T), \label{Pi1}
    \\
    &\int_T \Pi_T\omega = \int_T \omega  \label{Pi2}
\end{alignat}
\end{subequations}
for any $\omega$ in
$D_{\Pi_T}:=\{ \omega \in H(\divg, T, \mathbb{S}): \;
\omega \bn|_F \in [L^2(F)]^3, \text{ for all } F \in \triangle_2(T) \}$. 
It follows that
\begin{equation}\label{commutePi}
  \int_T \divg ( \Pi_T \omega -\omega) \cdot v =0
  \quad {\text{ for all }} v\in V_h(T),\; \omega \in D_{\Pi_T}
\end{equation}
by integrating by parts and \eqref{Pi}.

The global discrete spaces are given by
\begin{alignat}{1}
    \Sigma_h:=&\{ \omega \in H(\divg, \Omega, \mathbb{S}): \omega|_T \in  \Sigma_h(T), \text{ for all } T \in \Th\}, \\
    V_h:=&  \{ v \in [L^2(\Omega)]^3: v|_T \in V_h(T),  \text{ for all } T \in \Th\}, \\
    W_h:=&   \{ w \in [L^2(\Omega)]^3: w|_T \in W_h(T),  \text{ for all } T \in \Th\}.
\end{alignat}
We also define the global interpolant $\Pi: D_{\Pi} \mapsto \Sigma_h$
by $\Pi \omega|_T = \Pi_T \omega$ for all $T \in \Th$ where
\begin{align}
  \label{eq:DPi-defn}
    D_{\Pi}:=\{ \omega \in H(\divg, \Omega, \mathbb{S}): \omega \bn|_F
\in [L^2(F)]^3, {\text{ for all }} F \in \triangle_2(\Th) \}.     
\end{align}
Note that
$H^1(\Omega, \mathbb{S}) \subset D_{\Pi}$. Finally, we need
the global versions of $P_T$ and $I_T$: the projection $P$ onto $W_h$
and $I: W_h \mapsto V_h$ are defined by $P v|_T=P_T v$ and
$I w|_T=I_T w$ for all $T \in \Th$.
The next result is an easy consequence of~\eqref{commutePi}.

\begin{proposition}
  For all $v \in V_h$ and $\og \in D_\Pi$, 
\begin{equation}\label{commutePiglobal}
  ( \divg(\Pi \omega -\omega),  v) =0.
\end{equation}
Moreover, for $s=1,2$,  for all $\og \in H^s(\om, \mathbb S)$, 
\begin{alignat}{2} \label{Pibound}
  \| \omega-\Pi \omega\|_{L^2(\Omega)}+ h \|\divg(\omega-\Pi \omega)\|_{L^2(\Omega)}
  & \le  C h^s \|\omega\|_{H^s(\Omega)}.
  % \qquad && \text{ for all } \omega \in  H^s(\Omega, \mathbb{S}).
\end{alignat}
\end{proposition}

The next result follows easily by combining a well-known technique of
proving wellposedness~\cite{Falk08} with a more recent result on
regular potentials~\cite[Corollary~4.7]{costabel2010bogovskiui} that show that 
given any $z \in H^\ell(\om)$, for any real $\ell$, there exists a
$u \in [H^{\ell+1}(\om)]^3$ such that
\begin{equation}
  \label{eq:CM}
  \divg u = z, \qquad \| u \|_{H^{\ell+1}(\om)} \le C \| z \|_{H^\ell(\om)}.
\end{equation}
Hereon, $C$ will denote a generic mesh-independent constant whose
value at different occurrences may vary.

\begin{lemma}
  \label{lem:inf-sup}
  Given any $v \in [L^2(\om)]^3$, there exists a $\tau \in H^1(\om, \SSS)$ such that
  \begin{equation}
    \label{eq:div-S-rgt-inv}
    \divg \tau
    = v, \qquad 
  \|\tau\|_{H^1(\Omega, \SSS)} \le C \| v\|_{L^2(\Omega)}.    
  \end{equation}
\end{lemma}
\begin{proof}
  Using~\eqref{eq:CM} with $\ell=0$ for each
  component of $v$, we find a $\tau_1 \in H^1(\om, \M)$ such that
  \begin{equation}
    \label{eq:CM-1}
    \divg \tau_1 = v, \qquad \| \tau_1 \|_{H^1(\om, \M)} \le C \| v \|_{L^2(\om)}.
  \end{equation}
  Next, we apply~\eqref{eq:CM} with $\ell=1$ and $z=- 2 \vskw \tau_1$
  to find a $\rho$ such that $\divg (\Xi \rho) = - 2 \vskw
  \tau_1$. Since $\vskw \tau_1 \in [H^1(\om)]^3$, and $\Xi$ is a
  pointwise algebraic bijection, the resulting $\rho$ is in
  $H^2(\om, \M)$ and
  \begin{equation}
    \label{eq:CM-2}
    \| \rho \|_{H^2(\om)} \le C \| \tau_1 \|_{H^1(\om)}.
  \end{equation}
  Then putting
  $\tau = \tau_1 + \curl \rho$ we find, using \eqref{alg1-1}, that
  $ \divg \tau = v $ and
  \[
    2\vskw \tau =
    2\vskw \tau_1 + 2\vskw (\curl \rho) = 2\vskw \tau_1 + \divg (\Xi \rho) = 0. 
  \]
  Combining~\eqref{eq:CM-1} and~\eqref{eq:CM-2}, we conclude that
  $\tau = \tau_1 + \curl \rho$ is in
  $H^1(\Omega, \mathbb{S})$, satisfies $\divg \tau = v$, as well as 
  the estimate in~\eqref{eq:div-S-rgt-inv}.
\end{proof}

\begin{theorem}[Discrete stability]\label{thm:inf-sup-JM}
There exists a constant $\beta>0$ such that 
\begin{equation}
   \beta \| v \|_{L^2(\om)} \le \sup_{0 \neq \omega \in \Sigma_h}  \frac{( \divg \omega, v)}{\|\omega\|_{H(\divg,\Omega)} } \qquad {\text{ for all }} v \in V_h. 
\end{equation}
\end{theorem}
\begin{proof}
  Given a $v\in V_h$, we apply Lemma~\ref{lem:inf-sup} to get  a $\tau$ in
  $H^1(\om, \mathbb S)$ satisfying~\eqref{eq:div-S-rgt-inv}.  Since
  $H^1(\om, \mathbb S) \subset D_\Pi$, the function 
  $\omega= \Pi \tau$ exists in $\Sigma_h.$ Then we have
\begin{equation}
\|v\|_{L^2(\Omega)}^2= (\divg \tau, v)= (\divg \omega, v) 
\end{equation}
by \eqref{commutePiglobal}. The result now follows after using \eqref{Pibound} with $s=1$ 
 and~\eqref{eq:div-S-rgt-inv} which gives $\|\omega\|_{H(\divg, \Omega)} \le C \|\tau\|_{H^1(\Omega)} \le C \| v\|_{L^2(\Omega)}$. 
\end{proof}

{\bf The mixed method with JKM element} finds 
$\sigma_h \in \Sigma_h$ and $u_h \in V_h$ satisfying
\begin{subequations}\label{FEM}
\begin{alignat}{2}
(\Aa \sigma_h, \tau)+(u_h, \divg \tau)=& 0 \qquad && {\text{ for all }} \tau \in \Sigma_h, \label{FEM1}\\
(\divg \sigma_h, v)=& (f,v) \qquad  && {\text{ for all }} v \in V_h. \label{FEM2}
\end{alignat}
\end{subequations}

\subsection{Error analysis}

We prove error estimates for~\eqref{FEM}.
Let $\| \tau \|_{\Aa} = (\Aa \tau, \tau)^{1/2}$.

\begin{theorem}\label{mainerrorthm}
  Let $\sigma, u$ solve \eqref{Elas}
  and $\sigma_h, u_h$ solve \eqref{FEM}.
  Then, we have 
\begin{equation}\label{error:sigma}
\|\Pi \sigma-\sigma_h\|_\Aa \le \| \Pi \sigma-\sigma\|_\Aa
\end{equation}
and
\begin{equation}\label{error:u}
  \|u-u_h\|_{L^2(\Omega)} \le C \big(\| \sigma-\sigma_h\|_\Aa
  + \|Q u-u\|_{L^2(\Omega)}\big)
\end{equation}
where  $Q$ is the $L^2$-orthogonal projection onto $V_h$.
\end{theorem}
\begin{proof}

Using \eqref{commutePiglobal}, \eqref{FEM2} and \eqref{Elas2} we get that 
\begin{equation*}
(\divg (\Pi \sigma-\sigma_h), v)=0  \quad {\text{ for all }} v \in V_h.
\end{equation*}
However, $(\divg (\Pi \sigma-\sigma_h), P v)= (\divg (\Pi \sigma-\sigma_h), v)$. Since $P: V_h \rightarrow W_h$ is an isomorphism, this implies that
\begin{equation}\label{divsig}
\divg (\Pi \sigma-\sigma_h)=0.
\end{equation}
Then, \eqref{error:sigma} follows from \eqref{FEM1} and \eqref{Elas1}.

Next, using Lemma~\ref{lem:inf-sup}, we obtain a
$\tau \in H^1(\Omega, \mathbb{S})$ satisfying 
\begin{align}
  \divg \tau & = Q u-u_h,  \qquad  && \label{aux107.1}   \\
  \|\tau\|_{H^1(\Omega)} & \le  C \| Qu-u_h\|_{L^2(\Omega)}. \label{aux107.2}  
\end{align}
Combining~\eqref{aux107.1} with   \eqref{commutePiglobal}, 
\begin{align*}
  \| Qu-u_h\|_{L^2(\Omega)}^2&= (Qu-u_h, Qu-u_h)
                               = (u-u_h, Qu - u_h)
  \\
    % &= (u-u_h, Qu - u_h) && \text{(property~of~$Q$, $Qu - u_h \in V_h$)} \\
    &= (u-u_h, \divg \Pi \tau)+(u-u_h, \divg(\tau-\Pi \tau)) && \text{(by~\eqref{aux107.1})} \\
    &=-(\Aa(\sigma - \sigma_h), \Pi \tau) + (u-u_h, \divg(\tau-\Pi \tau)) && \text{(by~\eqref{Elas1}, \eqref{FEM1})} \\
    &=-(\Aa(\sigma - \sigma_h), \Pi \tau) + (u-Qu, \divg(\tau-\Pi \tau)) && \text{(by~\eqref{commutePiglobal}).}
\end{align*}
By the Cauchy--Schwarz inequality, \eqref{Pibound}, and \eqref{aux107.2}, 
\begin{equation*}
\|Qu-u_h\|_{L^2(\Omega)} \le C \big(\| \sigma-\sigma_h\|_{\Aa}+ \|Q u-u\|_{L^2(\Omega)}\big),
\end{equation*}
so \eqref{error:u} now follows by the triangle inequality. 
\end{proof}

Using the above result along with \eqref{Pibound} and the approximation properties of $Q$ we obtain the following corollary.
\begin{corollary}
  \label{cor:errorin-sig_u}
Under the hypothesis of Theorem \ref{mainerrorthm}, assuming also that $\sigma \in H^2(\Omega, \mathbb{S})$ and $u \in [H^2(\Omega)]^3$, we have 
\begin{equation}
  \|\sigma-\sigma_h\|_{\Aa}
  + \|u-u_h\|_{L^2(\Omega)}
  \le C  h^2 (\|\sigma\|_{H^2(\Omega)}+ \|u\|_{H^2(\Omega)}).
\end{equation}
\end{corollary}

\subsection{Superconvergence and postprocessing}
In this subsection, we prove that a projection of the error $u-u_h$ superconverges at a rate
$h^3$ if we assume full elliptic regularity, i.e., if we assume that $\om$ and $\Aa$ are such that the inequality
\begin{equation}
  \label{eq:full-reg}
  \| u \|_{H^2(\om)} + \| \sigma \|_{H^1(\om)} \le C \| f \|_{L^2(\om)}
\end{equation}
holds for any solution $\sigma, u$ of \eqref{strong-eqs} obtained with an 
$f \in L^2(\om)$.
We will need the
following commuting property:
\begin{equation}\label{commutePiglobalw}
  % ( \divg(\Pi \omega -\omega),  w) =0,
  % \qquad {\text{ for all }} w\in W_h, \;
  % \omega \in    H^1(\Omega, \mathbb{S}), \;
  % \divg \omega \in W_h. 
  \divg(\Pi \omega -\omega) = 0,
  \qquad 
  \omega \in    H^1(\Omega, \mathbb{S}), \;
  \divg \omega \in W_h. 
\end{equation}
Indeed, for such $\omega$, since  $\divg(\Pi \omega -\omega) \in W_h$, for all $w \in W_h,$
we have   $( \divg(\Pi \omega -\omega),  w) =( \divg(\Pi \omega -\omega),  I w)$ by \eqref{ITinner}, and the latter vanishes by \eqref{commutePiglobal}.

\begin{theorem}
  \label{thm:errorprojsuperconv}
  Let $\sigma, u$ solve \eqref{Elas} and $\sigma_h, u_h$ solve \eqref{FEM}. Assume that the full regularity estimate~\eqref{eq:full-reg} holds.
  Then, the following estimate holds
\begin{equation}
  \label{Pu-estimate}
    \| P(u-u_h)\|_{L^2(\Omega)} \le   C h
    \| \sigma-\sigma_h\|_{\Aa} + C h^2 \| \divg(\sigma-\Pi \sigma)\|_{L^2(\Omega)}. 
  \end{equation}
  If  $\sigma \in H^2(\om)$, this implies
  \begin{equation}
    \label{Pu-estimate-2}
    \| P(u-u_h)\|_{L^2(\Omega)} \le   C h^3     |\sigma|_{H^2(\om)}.
  \end{equation}
\end{theorem}
\begin{proof}
Let $\tau$ and $\psi$ solve
\begin{subequations}\label{dual}
\begin{alignat}{2}
\Aa \tau=& \epsilon (\psi) \qquad && \text{ in } \Omega,  \label{dual1}\\
\divg \tau=& P(u-u_h) \qquad  && \text{ in } \Omega, \label{dual2}\\
\psi=& 0 \qquad && \text{ on }  \partial \Omega. \label{dual3}
\end{alignat}
\end{subequations}
By the full elliptic regularity assumption~\eqref{eq:full-reg}, we have
\begin{equation}\label{ellipticreg}
\| \tau\|_{H^1(\Omega)}+ \|\psi\|_{H^2(\Omega)} \le C \| P(u-u_h)\|_{L^2(\Omega)}.
\end{equation}
Then, 
\begin{alignat*}{2}
\| P(u-u_h)\|_{L^2(\Omega)}^2& = ( P(u-u_h), \divg  \tau)  \qquad  && \text{(by \eqref{dual2})} \\
& =  ( P(u-u_h), \dive  \Pi \tau) \qquad  && \text{(by \eqref{commutePiglobalw})}  \\
& =  ( u-u_h, \dive  \Pi \tau) \qquad % && \text{(by property of $P$)} 
  \\
& =  -( \mathcal{A}(\sigma-\sigma_h), \Pi \tau)  \qquad && \text{(by \eqref{FEM1}, \eqref{Elas1})}  \\
& = -( \mathcal{A}(\sigma-\sigma_h), \Pi \tau- \tau)-( \mathcal{A} (\sigma-\sigma_h),  \tau).
\end{alignat*}
We can now simplify the last term
\begin{alignat*}{2}
    ( \mathcal{A} (\sigma-\sigma_h),  \tau)& =  ( \sigma-\sigma_h, \mathcal{A} \tau) \qquad && \text{(by symmetry of $\Aa$)} \\
    & =  (\sigma-\sigma_h, \epsilon(\psi)) \qquad && \text{(by \eqref{dual1})} \\
    & =  -( \dive(\sigma-\sigma_h), \psi) \qquad && \text{(integration by parts)}  \\
    & =  -( \dive(\sigma-\Pi \sigma), \psi-Q \psi ) \qquad && \text{(by \eqref{divsig}, \eqref{commutePiglobal}).} 
\end{alignat*}
Hence, we obtain 
\begin{equation*}
    \| P(u-u_h)\|_{L^2(\Omega)}^2=  -( \mathcal{A}(\sigma-\sigma_h), \Pi \tau- \tau)+ ( \dive(\sigma-\Pi \sigma), \psi-Q \psi ).
\end{equation*}
The first estimate \eqref{Pu-estimate}  of the theorem now follows after using \eqref{Pibound}, approximation properties of $Q$, and \eqref{ellipticreg}. The second follows from the first and \eqref{Pibound}.
\end{proof}

We can use the superconvergence result to construct a better approximation to $u$ using a local post-processing scheme. % We start by defining 
% \begin{align*}
%     V_h^2:= \{ v \in  [L^2(\Omega)]^3: v|_T \in [\pol_2(T)]^3, {\text{ for all }} T \in \Th \}.
% \end{align*}
Let $\RM(T)$ denote the space of  rigid displacements on $T$. Then, we can decompose $W_h(T)$ as 
\begin{equation}
  \label{eq:Wh-decomp}
    W_h(T)= P_T \RM(T)+ (P_T \RM(T))^{\perp}.
\end{equation}
where % Note that more explicitly these space are 
\begin{align*}
P_T \RM(T)&:= \{ P_T v: v \in \RM(T) \}=\{ w \in W_h(T): I_T w \in \RM(T)\}, \\
(P_T \RM(T))^{\perp}&:= \{ w \in W_h(T): (w, P_T v)_T=0, \text{ for all } v \in \RM(T) \}.
\end{align*}
Define the projection $P_T^\RM:L^2(T) \to P_T\RM(T)$ by 
\begin{equation}
  (P_T^\RM q, w)_T=(q, w)_T, \quad {\text{ for all }} w \in P_T\RM(T).
\end{equation}
Recalling that $P_T$ is injective on $V_h(T)$  due to~\eqref{eq:PhIh}, 
we have % we see that the $\dim  P_T \RM(T)= 6.$
% and consequently  $\dim  (P_T \RM(T))^{\perp}= 6$.
% The injectivity of $P_T$ also proves 
\begin{subequations}\label{enorms}
\begin{alignat}{2}
  \|v\|_{L^2(T)} & \le  C \| P_T v\|_{L^2(T)}  \qquad  && {\text{ for all }} v \in V_h(T),  \label{enorms1} \\
  \intertext{by 
equivalence of norms on finite dimensional spaces and scaling. Similarly,}
\|q\|_{L^2(T)} & \le C  h_T \| \epsilon(q)\|_{L^2(T)}+ \| P_T^\RM q \|_{L^2(T)}  \qquad  && {\text{ for all }} q \in [\pol_2(T)]^3. \label{enorms2}
\end{alignat}
\end{subequations}
Indeed, to show that the right hand side of \eqref{enorms2} is a norm on $[\pol_2(T)]^3,$ suppose  $\epsilon(q)=0$ on $T$.  Then $q$ must be a rigid displacement.  If furthermore $P_T^\RM q=0,$  then we see that $P_T q=0$  and \eqref{enorms1} shows that $q=0$. Together with standard scaling arguments, this establishes~\eqref{enorms2}.
Define the quadratic space
\begin{align}
  \label{eq:Sh-defn}
    S_h(T):=\{ q \in [\pol_2(T)]^3: (q,v)_T=0, {\text{ for all }} v \in \RM(T) \}.
\end{align}
Using it, and following Stenberg \cite{stenberg1991postprocessing}, we define our post-processed approximation next.

{\bf{The post-processing system}}
finds $\ust_h \in [\pol_2(T)]^3$ satisfying
\begin{subequations}\label{post}
\begin{alignat}{2}
(\epsilon(\ust_h), \epsilon(v))_T=& (\Aa \sigma_h, \epsilon(v))_T \qquad && {\text{ for all }} v \in S_h(T), \label{post1}\\
(\ust_h, w)_T=& (Pu_h,w)_T \qquad  && {\text{ for all }} w \in P_T \RM(T), \label{post2}
\end{alignat}
\end{subequations}
for every $T \in \Th$, given $\sigma_h, u_h$ solving~\eqref{FEM}.    It is obvious from the
definition~\eqref{eq:Sh-defn} of $S_h(T)$ that
$\dim S_h(T) + \dim \RM(T) = \dim [\pol_2(T)]^3$.  Since
\begin{equation}
  \label{eq:6}
  \dim \RM(T) = \dim P_T \RM(T)  
\end{equation}
by the injectivity of  $P_T$ on $V_h(T)$, 
the system~\eqref{post} is
square. Moreover, \eqref{post} is uniquely solvable: indeed 
if its right hand side vanishes,
then, \eqref{post1} shows that $\ust_h$ is a rigid displacement, which then together with \eqref{post2} with \eqref{enorms2} implies that $\ust_h$ must vanish.

Next, we show that $u - u_h^\star$ converges at a rate higher than
that expected for $u-u_h$ from Corollary~\ref{cor:errorin-sig_u}.  Let
$Q_2$ denote the element-wise $L^2$-orthogonal projection onto
$[\pol_2(T)]^3$.

\begin{theorem}\label{mainthmpost}
Let $\sigma, u$ solve \eqref{Elas} and $\sigma_h, u_h$ solve \eqref{FEM}. If $\ust_h$ is given by \eqref{post}, then
\begin{alignat*}{1}
  \|u-\ust_h\|_{L^2(T)} & \le  C (\|u-Q_2 u\|_{L^2(T)}+\|P_T(u_h- u)\|_{L^2(T)})\\
  & + C h_T \big(\|\epsilon(u-Q_2u)\|_{L^2(T)}+ \|\Aa(\sigma-\sigma_h)\|_{L^2(T)}\big). 
\end{alignat*}
\end{theorem}

\begin{proof}
Using \eqref{post1} and \eqref{Elas1} we obtain
\begin{alignat}{2}
(\epsilon(\ust_h-u), \epsilon(v))_T=& (\Aa (\sigma_h-\sigma), \epsilon(v))_T  \qquad \text{ for all } v \in S_h(T).
\end{alignat}
Since the same equality trivially holds for all $v \in \RM(T)$, it actually
holds for all $v \in [\pol_2(T)]^3$. Hence we may choose $v = u_h^\star - Q_2u$ in
$[\pol_2(T)]^3$ and manipulate to get 
\begin{equation}
  \label{eq:2}
\|\epsilon(\ust_h-Q_2 u)\|_{L^2(T)} \le  \|\epsilon(u-Q_2u)\|_{L^2(T)}+ \|\Aa(\sigma-\sigma_h)\|_{L^2(T)}.
\end{equation}
Using \eqref{enorms2} we have
\begin{equation}
  \label{eq:4}
  \|\ust_h-Q_2u\|_{L^2(T)} \le C h_T \|\epsilon(\ust_h-Q_2u)\|_{L^2(T)}+ C \| P_T^\RM (\ust_h-Q_2u)\|_{L^2(T)}. 
\end{equation}
It is easy to see that $P_T^\RM Q_2 u = P_T^\RM P_T Q_2 u$. Also,
\eqref{post2} implies $P_T^\RM u_h^\star = P_T^\RM P_T u_h$. 
% However, $P_T^\RM (\ust_h-Q_2u)= P_T^\RM (u_h -Q_2 u)=  P_T^\RM P(u_h-Q_2 u)$. 
Therefore, 
\begin{align}
  \nonumber 
  \| P_T^\RM (\ust_h-Q_2u)\|_{L^2(T)}  & \le \| P_T(u_h-Q_2 u)\|_{L^2(T)}
  \\ \label{eq:3}
  & \le \| P_T(u_h- u)\|_{L^2(T)}+ \| u- Q_2 u\|_{L^2(T)}.
\end{align}
The result now follows by using the estimates of~\eqref{eq:2}
and~\eqref{eq:3} within~\eqref{eq:4}.
\end{proof}

% One immediately obtains the following corollary by Theorem~\ref{mainthmpost}, \eqref{Pu-estimate}, and \eqref{Pibound}.

\begin{corollary}
  Under the hypotheses of
  Theorems~\ref{thm:errorprojsuperconv} and~\ref{mainthmpost}, 
  assuming also that $\sigma \in H^2(\Omega, \mathbb{S}), u \in [H^3(\Omega)]^3$,
  we have 
\begin{alignat*}{1}
  \|u-\ust_h\|_{L^2(\Omega)} \le   C h^3 (\|\sigma\|_{H^2(\Omega)}+ \|u\|_{H^3(\Omega)}).
\end{alignat*}
\end{corollary}
\begin{proof}
  Apply \eqref{Pu-estimate}, \eqref{Pibound}, and a standard estimate
  for projection error of $Q_2$, to further bound the estimate for
  $\|u-\ust_h\|_{L^2(\Omega)} $ given by Theorem~\ref{mainthmpost}.
\end{proof}

\subsection{Robustness in the incompressible limit}
  
We now investigate the convergence of the method for nearly
incompressible isotropic materials. Recall that $\Aa$ remains bounded
in this case but may become close to singular, so methods using
$\Aa^{-1}$ can become problematic. One of the advantages of the mixed
method is that it only needs $\Aa$. Nonetheless, as
$\| \cdot \|_\Aa$-norm becomes weaker in the incompressible limit, it
is natural to ask if $\sigma_h$ converges in a material-independent
norm, a question not answered by Corollary~\ref{cor:errorin-sig_u}.  To
answer this, we now use techniques similar to those in~\cite{arnold1984family}  to  prove that the $L^2$ error
$\|\sigma - \sigma_h\|_{L^2(\Omega)}$ has an error bound that is
robust for nearly incompressible materials satisfying the weaker
ellipticity inequality~\eqref{A-incompressible}.  In particular, for isotropic  materials with $\lambda \to \infty$, the $C$ below does not depend on~$\lambda$.
  
\begin{theorem}
  \label{thm:ee-robust}
  Assume that only \eqref{A-incompressible} holds (instead of~\eqref{A-coercive}).
  Let $\sigma, u$ solve \eqref{Elas} and $\sigma_h, u_h$ solve \eqref{FEM}. Then, 
    \begin{align}\label{error:sigma-incompressible}
        \|\sigma-\sigma_h\|_{L^2(\Omega)} \le C ( \|\sigma-\sigma_h\|_{\Aa} + h \|\div (\sigma - \Pi \sigma) \|_{L^2(\Omega)}). 
    \end{align}
    % where $C>0$  only depends on $\theta$ in \eqref{A-incompressible}, the domain $\om$, and the shape-regularity of the mesh. % constants of Bramble--Hilbert 
    % lemma, and $\Omega$. 
\end{theorem}
\begin{proof}
%   Note that
% \begin{align*}
%   \|\sigma - \sigma_h\|_{L^2(\Omega)}
%   &\le
%     C(\|\dev (\sigma - \sigma_h)\|_{L^2(\Omega)}
%     + \| \tr(\sigma - \sigma_h) \|_{L^2(\Omega)})
%     \\
%     &\le C(\|\sigma - \sigma_h\|_{\Aa} + \| \tr(\sigma - \sigma_h) \|_{L^2(\Omega)}).
% \end{align*}
% Since $\|\sigma - \sigma_h\|_{\Aa}$ is estimated by \eqref{error:sigma}, it suffices to estimate $\|\tr(\sigma - \sigma_h) \|_{L^2(\Omega)}$.
It is evident by setting $\tau_h$ to the $3 \times 3$ identity matrix
$\mathbb{I}$ in \eqref{Elas} and \eqref{FEM} that
$\int_{\Omega}\tr (\sigma - {\sigma}_h) =0.$
Hence there exists $\phi \in [H_0^1(\Omega)]^3$  (see, e.g., \cite{costabel2010bogovskiui}) such that
\begin{equation}
  \label{eq:phi-tr}
  \dive \phi = \tr (\sigma - \sigma_h), \qquad
  \| \phi \|_{H^1(\Omega)} \le C \| \tr(\sigma - \sigma_h)\|_{L^2(\Omega)}.
\end{equation}
Using this, 
\begin{equation}
  \label{eq:tr-id}
\begin{aligned}
  % \nonumber
  \|\tr(\sigma - {\sigma}_h)\|_{L^2(\Omega)}^2
  &= (\tr(\sigma - {\sigma}_h), \div \phi) 
  =  ( \tr(\sigma - {\sigma}_h) \mathbb{I},  \epsilon(\phi))
  \\ % \notag
  % &= \int_{\Omega} ((\sigma - {\sigma}_h^R) - (\sigma - {\sigma}_h^R)^D) : \epsilon(\phi) 
  % \\
%   &= 3 \int_{\Omega} \big((\sigma - {\sigma}_h) - \dev(\sigma - {\sigma}_h)
% \big) : \epsilon(\phi) 
  &= 3 \big(\sigma - {\sigma}_h, \epsilon(\phi) \big)
 - 3(\dev(\sigma - {\sigma}_h),    \epsilon(\phi) \big)
  \\ % \notag
  &= \revj{-} 3(\dive (\sigma - {\sigma}_h),  \phi) - 3
    (\dev(\sigma - {\sigma}_h), \epsilon(\phi)). 
    %\\
    %&= \int_{\Omega} \dive(\sigma - {\sigma}_h^R) \cdot \phi - \int_{\Omega} ((\sigma - {\sigma}_h^R)^D: \epsilon(\phi) .   
  \\ 
  &= \revj{-} 3(\dive (\sigma - \Pi \sigma), \phi) - 3
    (\dev(\sigma - {\sigma}_h), \epsilon(\phi)),
\end{aligned}
\end{equation}
where we used \eqref{divsig} in the last step.
To estimate the first term on the final right hand side above, we again use  \eqref{divsig}
(and also \eqref{commutePiglobal}, which is implied by  \eqref{divsig}) to get 
\begin{align}
  \label{eq:incompressible-div}
  (\dive (\sigma - {\sigma}_h),  \phi)
  &= ( \dive (\sigma - \Pi \sigma) , \phi) % && \text{(by \eqref{divsig})}
  % \\
  % \notag 
  % &
    = (\dive (\sigma - \Pi \sigma), \phi - Q \phi) % && \text{(by \eqref{commutePiglobal})}
  \\
  \notag
                                                       &\le Ch \|\dive (\sigma - \Pi \sigma)\|_{L^2(\Omega)} \|\phi\|_{H^1(\Omega)}
  % \\
  % \notag
  %                                                      &\le Ch \|\dive (\sigma - \Pi \sigma)\|_{L^2(\Omega)}  \|\tr(\sigma - {\sigma}_h)\|_{L^2(\Omega)} .
\end{align}
The last term in~\eqref{eq:tr-id} is easily estimated by
\eqref{A-incompressible}:
\begin{align*}
  ( \dev(\sigma - {\sigma}_h), \epsilon(\phi))
  &\le
    C\| \sigma - {\sigma}_h\|_{\Aa} \| \phi \|_{H^1(\Omega)}.%   && \text{(by \eqref{A-incompressible})}
    % \\
    % &\le C\| \sigma - {\sigma}_h\|_{\Aa} \|\tr(\sigma - {\sigma}_h)\|_{L^2(\Omega)}. 
\end{align*}
Together with~\eqref{eq:phi-tr}, 
these estimates imply  that  $\|\tr(\sigma - {\sigma}_h)\|_{L^2(\Omega)} \le Ch \|\sigma\|_{H^1(\Omega)}$, so the conclusion follows from 
$C\|\sigma - \sigma_h\|_{L^2(\Omega)} \le
\|\dev (\sigma - \sigma_h)\|_{L^2(\Omega)}
    + \| \tr(\sigma - \sigma_h) \|_{L^2(\Omega)}$ and
\eqref{A-incompressible}.
\end{proof}

\begin{remark} \label{rem:incomproofcritical}
  The above  proof extends to other cases (as we shall see in later sections)
  as long as an analogue of 
  the critical ingredient \eqref{divsig} is available.
  Also note that to obtain convergence
  rates, it is enough to simply use~\eqref{Pibound} in
  \eqref{error:sigma-incompressible}.
\end{remark}

\section{Piecewise constant displacements on the refinement}
\label{sec:P0-disp}

In this section we show that we can replace the displacement space $V_h$ with $W_h$ and obtain very similar results.
This replacement is attractive since $\div \Sigma_h = W_h$ and the equilibrium equation can then be exactly satisfied for numerical solutions. Also, 
the stress and displacement elements are implemented with 
respect to the same refined mesh. 
However, a price to pay, as we shall see now, is that we are not able to prove an  $O(h^3)$ superconvergence result (like in \eqref{Pu-estimate-2}) unless the right hand side $f$ belongs to~$W_h$. %

We start by proving unisolvency of slightly different dofs
(cf.~\eqref{Sigmah}) for the same stress space $\Sigma_h$.  It is
useful to note at the outset that by \eqref{eq:Wh-decomp}
and~\eqref{eq:6},
\[
  \dim (P_T \RM(T))^\perp = 12 - 6,
\]
so the
number of dofs in \eqref{Sigmah2} and \eqref{Sigmahp2} both equal six.

\begin{theorem}
An element $\omega \in  \Sigma_h(T)$  is uniquely determined by the following dofs:
\begin{subequations}
\label{Sigmahp}
\begin{alignat}{4}
&\int_F \omega \bn \cdot \kappa, \qquad &&\kappa\in [\pol_1(F)]^3, F \in \triangle_2(T), \qquad && \dofcnt{(36 dofs)} \label{Sigmahp1}\\
&\int_T \dive \omega \cdot v , \qquad && v \in (P_T \RM(T))^{\perp}. \qquad && \dofcnt{($6$ dofs)} \label{Sigmahp2}
\end{alignat}
\end{subequations}
\end{theorem}

\begin{proof}
Let $\omega \in  \Sigma_h(T)$  and assume that the dofs in \eqref{Sigmahp} vanish. Then we must show that $\omega=0$. The dofs in \eqref{Sigmahp1} show that $\omega \bn=0$ on $\partial T$. Then, by integration by parts, we get $\dive \omega \in  (P_T \RM(T))^{\perp}$, and \eqref{Sigmahp2} shows that $\dive \omega=0$. Then, by Lemma~\ref{Sigmazero}, $\omega=0$.  
\end{proof}

{\bf The mixed method with piecewise constant displacement}
finds $\ssigma_h \in \Sigma_h$ and $\uu_h \in W_h$ satisfying
\begin{subequations}\label{FEMp}
\begin{alignat}{2}
(\Aa \ssigma_h, \tau)+(\uu_h, \dive \tau)=& 0 \qquad && {\text{ for all }} \tau \in \Sigma_h, \label{FEMp1}\\
(\dive \ssigma_h, v)=& (f,v) \qquad  && {\text{ for all }} v \in W_h. \label{FEMp2}
\end{alignat}
\end{subequations}
The next result is analogous to Theorem~\ref{thm:inf-sup-JM}, so we omit its proof. 
\begin{theorem}
There exists a constant $\beta>0$ such that 
\begin{equation}
   \beta \le \inf_{0\not = v \in W_h} \sup_{0 \neq \omega \in \Sigma_h}  \frac{( \dive \omega, v)}{\|\omega\|_{H(\dive,\Omega)} \|v\|_{L^2(\Omega)} }. 
\end{equation}
\end{theorem}
% \begin{proof}
%      Let $v \in W_h$ then there exists $\tau \in H^1(\Omega, \mathbb{S})$ such that 
%   \begin{alignat}{2}
%   \dive \tau= &v  \qquad  && \text{ on } \Omega \label{tau1p}\\
%   \|\tau\|_{H^1(\Omega)} \le & C \| v\|_{L^2(\Omega)}. \label{tau2p} \qquad &&
%   \end{alignat}
%   Let $\omega= \PPi \tau$ then we have 
% \begin{equation}
% \|v\|_{L^2(\Omega)}^2= (\dive \tau, v)= (\dive \omega, v), 
% \end{equation}
% by \eqref{commutePPiglobal} where we used that $\dive \tau \in W_h$. The result now follows after using \eqref{PPibound} with $s=1$ 
%  and \eqref{tau2p} which gives $\|\omega\|_{H(\dive, \Omega)} \le C \|\tau\|_{H^1(\Omega)} \le C \| v\|_{L^2(\Omega)}$. 
% \end{proof}

Our error analysis of the method~\eqref{FEMp} does not proceed as in the previous section, but rather
through  the following auxiliary problem
to find $\tsigma \in \Sigma$,  $\tu \in V $ such that
% since the
% commuting property \eqref{commutePPiglobal} does not hold for all
% solutions $\sigma$.
% % when the divergence of the matrix field is in $W_h$ we
% % cannot use the argument as above to get $O(h^2)$ estimate
% % for $\sigma$.
% Instead, we proceed  by introducing the following auxiliary problem:
\begin{subequations}\label{Elast}
\begin{alignat}{2}
(\Aa \tsigma, \tau)+(\tu, \dive \tau)=& 0 \qquad && {\text{ for all }} \tau \in \Sigma, \label{Elas1t}\\
(\dive \tsigma, v)=& (Pf,v) \qquad  && {\text{ for all }} v \in V. \label{Elas2t}
\end{alignat}
\end{subequations}
Now note that $\dive \tsigma=Pf \in W_h$. 

% A standard energy argument gives the following estimate. 

\begin{proposition}\label{prop142}
  Let $\sigma, u$ solve \eqref{Elas} and $\tsigma, \tu$ solve
  \eqref{Elast}. Then
\begin{equation}
\|\sigma-\tsigma\|_{\Aa}+\| u-\tu\|_{L^2(\Omega)}  \le C h \|f-Pf\|_{L^2(\Omega)}.
\end{equation}
\end{proposition}
\begin{proof}
  Equations~\eqref{Elas} and \eqref{Elast} imply
  \begin{alignat}{1}
    \label{eq:7}
 \|\sigma-\tsigma\|_{\Aa}^2=(f-Pf, u-\tu)= (f-Pf, (u-\tu)-P(u-\tu)) . 
\end{alignat}
However, using element-wise Poincare inequality (see, e.g., \cite{Bebendorf-poincare}) and then using Korn's first inequality (\cite[Theorem~6.3-4]{Ciarlet:elasticity-1}),
\begin{equation*}
\|(u-\tu)-P(u-\tu) \|_{L^2(\Omega)}\le C h \| \epsilon(u-\tu)\|_{L^2(\Omega)} = C h \|\Aa(\sigma-\tsigma) \|_{L^2(\Omega)}.
\end{equation*}
Using this in~\eqref{eq:7} after an application of Cauchy-Schwarz
inequality, the stated bound for $\|\sigma-\tsigma\|_{\Aa}$ follows.
The same bound for $\|u - \tu\|_{L^2(\om)}$ now follows from~\eqref{Elas1t}, which implies that $ (u - \tu, \div \tau) = (\Aa (\tsigma - \sigma), \tau ),$ after choosing a $\tau$ such that $\div \tau=  u - \tu$ by  Lemma~\ref{lem:inf-sup}.
\end{proof}

\begin{remark}
At this point, one may choose to proceed using the canonical
interpolant $J$ corresponding to the dofs in~\eqref{Sigmahp}. One can prove an analogue of \eqref{Pibound}
for $J$ and also that for any $\omega \in D_\Pi$ such that
$\dive \omega \in W_h$, we have $ (\dive( \PPi \omega -\omega), w) =0$
for all $w\in W_h,$  which is enough for an error analysis using $J$.
But we pursue the quicker alternative of using the same $\Pi$ from Section~\ref{sec:JM-element}.  
\end{remark}

\begin{remark}
\revj{Alternatively, one may choose to proceed following the interesting work of Lederer and Stenberg  \cite{Stenberg-2023} to obtain error estimates in a broken $H^1$-norm.  Advantages of their approach include less regularity assumptions. In our analysis below based on $\Pi$, we can also reduce the regularity assumptions (cf. \cite[Theorem~4.8]{Amrouche-etal:1998}, \cite[p.125]{Brezzi-Fortin:1991}) and get optimal order of approximation for low regularity solutions. We don't pursue this, and present a more traditional analysis for its simplicity. %, but refer the reader to~\cite{Stenberg-2023} if more refined estimates are needed.
}  
\end{remark}

\begin{theorem}\label{mainerrorthmp}
Let $\sigma, u$ solve \eqref{Elas} and $\ssigma_h, \uu_h$ solve \eqref{FEMp}. Then, 
\begin{equation}\label{error:sigmap}
  \| \Pi \sigma-\ssigma_h\|_{\Aa} \le
  \| \Pi \tsigma-\tsigma\|_{\Aa}+
  \|\Pi (\tsigma-\sigma)\|_{\Aa}
\end{equation}
and
\begin{equation}\label{error:up}
  C \|u-\uu_h\|_{L^2(\Omega)} \le 
  \|\sigma-\ssigma_h\|_{\Aa}
  + \|P u-u\|_{L^2(\Omega)}.
\end{equation}
\end{theorem}
\begin{proof}
  Subtracting~\eqref{FEMp} from~\eqref{Elast}, 
\begin{alignat}{2}
(\Aa (\tsigma-\ssigma_h), \tau)+(\tu-\uu_h, \dive \tau)=& 0 \qquad && {\text{ for all }} \tau \in \Sigma_h, \label{2-ee1} \\
(\dive (\tsigma-\ssigma_h), v)=& 0 \qquad  && {\text{ for all }} v \in W_h,  \label{2-ee2}
\end{alignat}
Since $\dive (\tsigma - \ssigma_h)$ is in  $ W_h$, equation~\eqref{2-ee2}
implies 
\begin{equation}\label{2-451}
\dive(\tsigma-\ssigma_h)=0.
\end{equation}
Also, using \eqref{commutePiglobalw}, 
\begin{equation}\label{2-451b}
    \dive (\Pi \tsigma-\tsigma)= 0 .
\end{equation}
Adding~\eqref{2-451} and~\eqref{2-451b}, we have
\begin{equation}
  \label{eq:2-451c}
  \dive (\Pi \tsigma-\ssigma_h)=0.   
\end{equation}
Hence, setting $\tau =\Pi \tsigma-\ssigma_h$ in~\eqref{2-ee1}, we immediately obtain 
\[
  \|\Pi \tsigma-\ssigma_h\|_{\Aa} \le
  \|\Pi \tsigma-\tsigma \|_{\Aa}.
\]
The estimate \eqref{error:up} follows the same argument as the proof of \eqref{error:u}.
\end{proof}

% \begin{theorem}\label{mainerrorthmp}
% Let $\sigma, u$ solve \eqref{Elas} and $\ssigma_h, \uu_h$ solve \eqref{FEMp}. Then, 
% \begin{equation}\label{error:sigmap}
%   \| \PPi \sigma-\ssigma_h\|_{\Aa} \le
%   \| \PPi \tsigma-\tsigma\|_{\Aa}+
%   \|\PPi (\tsigma-\sigma)\|_{\Aa}
% \end{equation}
% and
% \begin{equation}\label{error:up}
%   C \|u-\uu_h\|_{L^2(\Omega)} \le 
%   \|\sigma-\ssigma_h\|_{\Aa}
%   + \|P u-u\|_{L^2(\Omega)}.
% \end{equation}
% \end{theorem}
% \begin{proof}
%   Subtracting~\eqref{FEMp} from~\eqref{Elast}, 
% \begin{alignat}{2}
% (\Aa (\tsigma-\ssigma_h), \tau)+(\tu-\uu_h, \dive \tau)=& 0 \qquad && {\text{ for all }} \tau \in \Sigma_h, \label{ee1} \\
% (\dive (\tsigma-\ssigma_h), w)=& 0 \qquad  && {\text{ for all }} w \in W_h,  \label{ee2}
% \end{alignat}
% Since $\dive (\tsigma - \ssigma_h)$ is in  $ W_h$, equation~\eqref{ee2}
% implies 
% \begin{equation}\label{451}
% \dive(\tsigma-\ssigma_h)=0.
% \end{equation}
% Also, using \eqref{commutePPiglobal} we have 
% \begin{equation}\label{451b}
% (\dive (\PPi \tsigma-\tsigma), w)= 0 \qquad  {\text{ for all }} w \in W_h,     
% \end{equation}
% Adding~\eqref{451} and~\eqref{451b}, we have
% \begin{equation}
%   \label{eq:451c}
%   \dive (\PPi \tsigma-\ssigma_h)=0.   
% \end{equation}
% Hence, setting $\tau =\PPi \tsigma-\ssigma_h$ in~\eqref{ee1}, we immediately obtain 
% \begin{equation}
%   \|\PPi \tsigma-\ssigma_h\|_{\Aa} \le
%   \|\PPi \tsigma-\tsigma \|_{\Aa}.
% \end{equation}
% The estimate \eqref{error:up} follows the same argument as the proof of \eqref{error:u}.
% \end{proof}

\begin{corollary}
Let $\sigma, u$ solve \eqref{Elas}, $\tsigma, \tu$ solve \eqref{Elast} and $\ssigma_h, \uu_h$ solve \eqref{FEMp}. If  $\sigma, \tsigma \in H^2(\Omega, \mathbb{S}), u \in [H^1(\Omega)]^3$, then
\[
\|\sigma-\ssigma_h\|_{\Aa}+ h \|u-\uu_h\|_{L^2(\Omega)} \le C  h^2 (\|\sigma\|_{H^2(\Omega)}+ \|\tsigma\|_{H^2(\Omega)}+ \|f\|_{H^1(\Omega)}+ \|u\|_{H^1(\Omega)}).
\]
 \end{corollary}

 The next intermediate result gives an $O(h^3)$ superconvergence
 result for $P(\tu-\uu_h)$. It  will help us prove
 superconvergence of the projection of the error 
 $P(u-\uu_h)$.

\begin{lemma}
  \label{lem:aux-pw}
  Let  $\tsigma, \tu$ solve \eqref{Elast} and $\ssigma_h, \uu_h$ solve \eqref{FEMp}. Assume that the full regularity estimate~\eqref{eq:full-reg} holds.
  Then, the following estimate holds
\begin{equation}\label{713}
  \| P(\tu-\uu_h)\|_{L^2(\Omega)} \le   C
  h \| \tsigma-\ssigma_h\|_{\Aa}.
\end{equation}
\end{lemma}
\begin{proof}
  Let $\tau$ and  $\psi$ solve
  % \eqref{dual} with $P(\tu-\uu_h)$ in place of
  % $P(u-u_h) $ in \eqref{dual2}.  
\begin{subequations}\label{dualp}
\begin{alignat}{2}
\Aa \tau& = \epsilon (\psi) \qquad && \text{ in } \Omega,  \label{dual1p}\\
\dive \tau& = P(\tu-\uu_h) \qquad  && \text{ in } \Omega, \label{dual2p}\\
\psi& = 0 \qquad && \text{ on }  \partial \Omega. \label{dual3p}
\end{alignat}
\end{subequations}
Since we are assuming full $H^2$-elliptic regularity we have:
\begin{equation}\label{ellipticregp}
\| \tau\|_{H^1(\Omega)}+ \|\psi\|_{H^2(\Omega)} \le C \| P(\tu-\uu_h)\|_{L^2(\Omega)}.
\end{equation}
Then, 
\begin{alignat*}{2}
\| P(\tu-\uu_h)\|_{L^2(\Omega)}^2=  &( P(\tu-\uu_h), \dive  \tau)  \qquad  && \text{(by \eqref{dual2p})} \\
& =  ( P(\tu-\uu_h), \dive  \Pi \tau) \qquad  && \text{(by  \eqref{commutePiglobalw})}  \\
& =  ( \tu-\uu_h, \dive  \Pi \tau) \\ % \qquad && \text{(property of $P$)}  \\
& =  -( \mathcal{A}(\tsigma-\ssigma_h), \Pi \tau)  \qquad && \text{(by \eqref{2-ee1})}  \\
& =  -( \mathcal{A}(\tsigma-\ssigma_h), \Pi \tau- \tau)-( \mathcal{A} (\tsigma-\ssigma_h),  \tau).
\end{alignat*}
We can now simplify the last term
\begin{alignat*}{2}
( \mathcal{A} (\tsigma-\ssigma_h),  \tau)& = ( \tsigma-\ssigma_h, \mathcal{A} \tau) \qquad && \text{(by symmetry of  $\Aa$)} \\
& = - (\tsigma-\ssigma_h, \epsilon(\revj{\psi})) \qquad && \text{(by  \eqref{dual1p})} \\
& = ( \dive(\tsigma-\ssigma_h), \revj{\psi}), \qquad && \text{(by integration by parts)}  
\end{alignat*}
which vanishes by \eqref{2-451}. Hence, we obtain 
\begin{equation*}
    \| P(\tu-\uu_h)\|_{L^2(\Omega)}^2=  -( \mathcal{A}(\tsigma-\ssigma_h), \Pi \tau- \tau).
\end{equation*}
The result now follows after using \eqref{Pibound} and \eqref{ellipticregp}.
\end{proof}

% Using the above result gives a superconvergence result for $P(u-\uu_h)$ that is $O(h^2)$ for general $f$, but that is $O(h^3)$ when $f \in W_h$. 
%
% (Moved to ending par of this section).

\begin{corollary}
  \label{cor:pw-const-superconvergence}
Under the same hypothesis of the above theorem we have
\begin{equation}\label{817}
 \| P(u-\uu_h)\|_{L^2(\Omega)} \le C h^2 (\|\sigma\|_{H^2(\Omega)}+ \|\tsigma\|_{H^2(\Omega)}+ \|f\|_{H^1(\Omega)}+ \|u\|_{H^1(\Omega)}).
\end{equation}
If $f \in W_h$ then  
\begin{equation}\label{814}
 \| P(u-\uu_h)\|_{L^2(\Omega)} \le C h^3 (\|\sigma\|_{H^2(\Omega)}+ \|\tsigma\|_{H^2(\Omega)}+ \|f\|_{H^1(\Omega)}+ \|u\|_{H^1(\Omega)}).
\end{equation}
\end{corollary}

\begin{proof}
The estimate \eqref{817} follows from \eqref{713} and Theorem \ref{mainerrorthmp} and Proposition \ref{prop142}. If $f \in W_h$ then $u=\tu$ and $\sigma=\tsigma$ hence the result follows also from \eqref{713} and Theorem \ref{mainerrorthmp} and Proposition \ref{prop142}.
\end{proof}

\begin{theorem}
  \label{thm:const-displ-incomp}
    Assume that \eqref{A-incompressible} \revj{holds}. Let $\sigma, u$ solve \eqref{Elas}, and $\ssigma_h, \uu_h$ solve \eqref{FEMp}. Then, 
    \begin{align*}
      \|\sigma-\ssigma_h\|_{L^2(\Omega)} \le C ( \|\sigma-\ssigma_h\|_{\Aa}
      + h % \|\div \sigma - P \div \sigma \|_{L^2(\Omega)})
       \|f  - P f \|_{L^2(\Omega)}).
    \end{align*}
    % with $C>0$ depending only on $\theta$ in \eqref{A-incompressible}, the constants of Bramble--Hilbert lemma, and $\Omega$. 
\end{theorem}
\begin{proof}
  The proof of Theorem~\ref{thm:ee-robust}
  can be repeated with $\sigma$ and $\sigma_h$ replaced by $\tsigma$ and
  $\ssigma_h$, respectively, since the analogue of the critical
  ingredient there (see Remark~\ref{rem:incomproofcritical}) is available
  and given by \eqref{eq:2-451c}. Then we obtain an
  analogue of \eqref{error:sigma-incompressible}, 
  \begin{align}
    \label{eq:ssig-hatsig}
    C \|\tsigma-\ssigma_h\|_{L^2(\Omega)}
    & \le  \|\tsigma-\ssigma_h\|_{\Aa} +
      h \|\div (\tsigma - \Pi \tsigma) \|_{L^2(\Omega)}.
  \end{align}
  The last term vanishes by~\eqref{commutePiglobalw}.

  To complete the proof,  it suffices to estimate $\| \sigma - \tsigma\|_{L^2(\om)}$. 
  Let $\phi \in [H_0^1(\Omega)]^3$ 
  be as in  \eqref{eq:phi-tr} with $\sigma_h$ replaced by $\tsigma$.
  Then, just as in~\eqref{eq:tr-id}, we obtain
  \begin{align*}
    \| \tr(\sigma - \tsigma) \|_{L^2(\om)}
    & 
      = 3 (\dive( \sigma - \tsigma), \phi) -
      3 (\dev(\sigma - \tsigma), \epsilon(\phi)).
  \end{align*}
  Now, $(\dive( \sigma - \tsigma), \phi) = (f - Pf, \phi - P\phi)$,
  and by  \eqref{A-incompressible},  $ \| \dev(\sigma - \tsigma) \| \le C \| \sigma - \tsigma \|_{\Aa}$. Hence we obtain
  \[
    C \| \sigma - \tsigma \|_{L^2(\om)} \le  \| \sigma - \tsigma\|_{\Aa}
    + h\| f - Pf \|_{L^2(\om)}.
  \]
  The proof is finished by combining this with \eqref{eq:ssig-hatsig}.
\end{proof}

To summarize the results in this section, 
we have shown (in Theorem~\ref{mainerrorthmp}) optimal error estimates for both
$\ssigma_h$ and $\uu_h$ when $\uu_h$ is piecewise constant,
% superconvergence of order $O(h^3)$ for $P(\tu-\uu_h)$ (in
% Lemma~\ref{lem:aux-pw}),
superconvergence of order $O(h^2)$ for $P(u - \uu_h)$ for general data
$f$, and superconvergence of order $O(h^3)$ for $P(u - \uu_h)$ when
$f \in W_h$ (in Corollary~\ref{cor:pw-const-superconvergence}).
\revj{The improvements in displacement error when $f$ is in $W_h$
  are akin to prior results of~\cite[Theorem~4.1]{stenberg1988elasticity} and
\cite[Theorem~2.1]{stenberg1991postprocessing}.}
We
have also shown (in Theorem~\ref{thm:const-displ-incomp}) the robustness of the method in the incompressible limit.

\section{A reduced space pair}
\label{sec:reduced-element}

In this section we present a finite element method obtained by
reducing certain judiciously chosen dimensions of the JKM element pair.
We replace the space $V_h$ by the space of element-wise rigid displacements
and reduce  $\Sigma_h$  accordingly.  The resulting method has optimal error estimates. It is
also robust for nearly incompressible materials (a property that
we cannot prove for an even further reduced pair of spaces given in
Appendix~\ref{sec:2nd-reduced-element}).
\revj{The idea of using
  element-wise rigid displacements as the displacement space
  can also be found in previous works such as~\cite{PitkaStenb83} and \cite{christiansen2019finite}.}

On a facet with normal
$n$ we let $v_t = n \times ( v \times n)$ for vector fields $v$, while
for matrix fields $\og$, we let $\og_{nn} = \og n \cdot n$ and
$\og_{nt} = n \times ( \og n \times n) \equiv (\og n )_t$.
Define
\begin{alignat*}{3}
  \Sigma_h^R(T)
  & :=\{ \omega \in \Sigma_h(T)
  & : \; &  \dive \omega \in P_T \RM(T),\;
  % \\
  % & & &
        \og_{nt} \in \RM(F),
        {\text{ for all }} F \in \triangle_2(T)\},\\
  V_h^R(T)  & := \RM(T).
\end{alignat*}
Here $\RM(F)$ is the space of rigid body motions defined on $F$.
We see that there are $18=6+4 \times 3$ constraints imposed on $\Sigma_h(T)$ to obtain $\Sigma_h^R(T)$. The corresponding global spaces are denoted by $\Sigma_h^R$ and $V_h^R$.

\begin{theorem} \label{thm:ee-reduced}
An element $\omega \in  \Sigma_h^R(T)$  is uniquely determined by the following dofs:
\begin{subequations}
\label{SigmaRh}
\begin{alignat}{4}
  \label{SigmaRh1}
    &\int_F \omega_{nn} \, \kappa \, {\ds}, \qquad &&\kappa\in \pol_1(F), F \in \triangle_2(T), \qquad && \dofcnt{(12 dofs)} 
    \\ 
    &\int_F  \og_{n t} \cdot \kappa  \, {\ds}, \qquad &&  \kappa\in \RM(F), F \in \triangle_2(T), \qquad && \dofcnt{(12 dofs).} \label{SigmaRh2}
\end{alignat}
\end{subequations}
\end{theorem}
\begin{proof}
We see that $\dim \Sigma_h^R(T) \ge \dim \Sigma_h(T)- 18=24$ which is exactly the number of dofs in \eqref{SigmaRh}. Now suppose that $\omega \in \Sigma_h^R(T)$ and the  dofs \eqref{SigmaRh1}, \eqref{SigmaRh2} vanish. Then,  obviously,   $\omega \bn =0$ on $\partial T$.

By the definition of the reduced stress space, there is a  $v \in \RM(T)$
such that $P_T v= \dive \omega$. Combined with the vanishing trace $\omega \bn =0$ on $\partial T$, 
\begin{alignat*}{1}
\|\dive \omega\|_{L^2(T)}^2 & =  \int_T \dive \omega \cdot P_T v =\int_T \dive \omega \cdot v  \\
& =  \int_{\partial T} \omega n \cdot v  =  0.
\end{alignat*}
 This shows that $\dive \omega=0$. Then by Lemma~\ref{Sigmazero},  $\omega$ must  vanish. 
\end{proof}

The canonical interpolant of the dofs in~\eqref{SigmaRh}, denoted by 
$\Pi^R_T \omega \in \Sigma_h^R(T)$, satisfies 
\begin{subequations}
\label{PiR}
\begin{alignat}{4}
&\int_F (\Pi^R_T \omega)_{nn} \,  \kappa \, {\ds}= \int_F  \omega_{nn} \, \kappa \, {\ds} , \qquad &&\kappa\in \pol_1(F), F \in \triangle_2(T), \label{Pi1R}\\
&\int_F  (\Pi^R_T \omega)_{nt} \cdot \kappa  \, {\ds}= \int_F  \omega_{nt} \cdot \kappa  \, {\ds} \qquad && \kappa\in \RM(F), F \in \triangle_2(T),  \label{Pi2R}
\end{alignat}
\end{subequations}
and we denote  the corresponding global projection by $\Pi^R$.

\begin{lemma}
For any $\og \in D_\Pi$, 
\begin{equation}\label{commutePiglobalR}
 (\dive( \Pi^R \omega -\omega),  v) =0 \qquad {\text{ for all }} v \in V_h^R.
\end{equation}
\end{lemma}
\begin{proof}
Let $v \in \RM(T)$. Then, 
\begin{alignat*}{2}
(\dive( \Pi^R \omega -\omega), v)_T
& =  \int_{\partial T } (\Pi^R \omega -\omega) \bn \cdot v   \\
& =  \int_{\partial T } (\Pi^R \omega -\omega)_{nn}   \, v \cdot \bn+ \int_{\partial T } (\Pi^R \omega -\omega)_{nt} \cdot v_t  =  0,
\end{alignat*}
since  $v_t \in \RM(F)$. 
\end{proof}

The following two results are similar to their previous counterparts and
their proofs are omitted.
\begin{proposition} 
 For all $\og \in H^1(\om, \mathbb S), $
    \begin{alignat}{2} \label{PiboundR}
      \| \omega-\Pi^R \omega\|_{L^2(\Omega)}+ h \|\dive(\omega-\Pi^R \omega)\|_{L^2(\Omega)} & \le   C h \|\omega\|_{H^1(\Omega)}.
    \end{alignat}
\end{proposition}
\begin{theorem}  
    There exists a constant $\beta>0$ such that 
   \[
      \beta \le \inf_{0\not = v\in V^R_h} \sup_{0 \neq \omega \in \Sigma_h^R}  \frac{( \dive \omega, v)}{\|\omega\|_{H(\dive,\Omega)} \|v\|_{L^2(\Omega)} } . 
    \]
\end{theorem}
% \begin{proof}
%      Let $v \in V_h^R$ then there exists $\tau \in H^1(\Omega, \mathbb{S})$ such that 
%   \begin{alignat}{2}
%   \dive \tau= &v  \qquad  && \text{ on } \Omega \label{tau1R}\\
%   \|\tau\|_{H^1(\Omega)} \le & C \| v\|_{L^2(\Omega)}. \label{tau2R} \qquad &&
%   \end{alignat}
%   Let $\omega= \Pi^R \tau$ then we have 
% \begin{equation}
% \|v\|_{L^2(\Omega)}^2= (\dive \tau, v)= (\dive \omega, v), 
% \end{equation}
% by \eqref{commutePiglobalR}. The result now follows after using \eqref{PiboundR} 
%  and \eqref{tau2R} which gives $\|\omega\|_{H(\dive, \Omega)} \le C \|\tau\|_{H^1(\Omega)} \le C \| v\|_{L^2(\Omega)}$. 
% \end{proof}

{\bf{The mixed method with the reduced space pair}} finds $\sigma^R_h \in \Sigma^R_h$ and $u^R_h \in V^R_h$ satisfying
\begin{subequations}\label{FEMR}
\begin{alignat}{2}
(\Aa \sigma^R_h, \tau)+(u^R_h, \dive \tau)& =  0 \qquad && {\text{ for all }} \tau \in \Sigma^R_h, \label{FEM1R}\\
(\dive \sigma^R_h, v)& =  (f,v) \qquad  && {\text{ for all }} v \in V^R_h. \label{FEM2R}
\end{alignat}
\end{subequations}
We  present an {\it a priori} error estimate for this method.

\begin{theorem}\label{mainerrorthmR}
Let $\sigma, u$ solve \eqref{Elas} and $\sigma^R_h, u^R_h$ solve \eqref{FEMR} then the following holds
\begin{equation}\label{error:sigmaR-Pi}
\|\Pi^R \sigma-\sigma^R_h\|_{\Aa}\le \|\Pi^R \sigma-\sigma\|_{\Aa},
\end{equation}
and
\begin{equation}\label{error:uR-P}
  \|u-u^R_h\|_{L^2(\Omega)} \le C \big(\| \sigma-\sigma^R_h\|_{\Aa}
  + \|Q^R u-u\|_{L^2(\Omega)}\big)
\end{equation}
where  $Q^R: [L^2(\Omega)]^3 \mapsto V^R_h$ is the $L^2$-orthogonal projection onto $V^R_h$.
\end{theorem}
\begin{proof}

Using \eqref{commutePiglobalR}, \eqref{FEM2R} and \eqref{Elas2} we get that 
\begin{equation*}
(\dive (\Pi^R \sigma-\sigma^R_h), v)=0  \quad {\text{ for all }} v \in V^R_h.
\end{equation*}
However, $(\dive (\Pi^R \sigma-\sigma^R_h), P v)= (\dive (\Pi^R \sigma-\sigma^R_h), v)$. Recalling~\eqref{eq:PhIh},  since $P_T: \RM(T) \rightarrow P_T \RM(T)$ is
bijection, we conclude that  
\begin{equation}\label{divsigR}
    \dive (\Pi^R \sigma-\sigma^R_h)=0.
\end{equation}
Then, \eqref{error:sigmaR-Pi} follows from \eqref{FEM1R} and \eqref{Elas1}. To prove \eqref{error:uR-P} one follows the same lines as the proof of~\eqref{error:u}. We leave the details to the reader. 
\end{proof}

\begin{corollary}
Let $\sigma, u$ solve \eqref{Elas}  and $\sigma^R_h, u^R_h$ solve \eqref{FEMR}. If  $\sigma  \in H^1(\Omega, \mathbb{S})$ and $u \in [H^1(\Omega)]^3$, then
 \begin{align}
    \|\sigma-\sigma^R_h\|_{\Aa} &\le C  h \|\sigma\|_{H^1(\Omega)} , \label{error:sigmaR} 
    \\
    \|u-u^R_h\|_{L^2(\Omega)} &\le C  h (\|\sigma\|_{H^1(\Omega)} +\|u\|_{H^1(\Omega)}). \label{error:uR}
\end{align}
\end{corollary}

\begin{theorem}
  \label{thm:ee-robustR}
  Assume that only \eqref{A-incompressible} holds.
  Let $\sigma, u$ solve \eqref{Elas} and $\sigma_h^R, u_h^R$ solve \eqref{FEM}. Then, 
    \begin{align}\label{error:sigmaR-incompressible}
        \|\sigma-\sigma_h^R\|_{L^2(\Omega)} \le C ( \|\sigma-\sigma_h^R\|_{\Aa} + h \|\div (\sigma - \Pi^R \sigma) \|_{L^2(\Omega)}).
    \end{align}
    % with $C>0$ which only depends on $\theta$ in \eqref{A-incompressible}, the constants of Bramble--Hilbert lemma, and $\Omega$. 
\end{theorem}

\begin{proof}
  The proof proceeds along the same lines
  as that of Theorem~\ref{thm:ee-robust} since the
  analogue of the critical ingredient mentioned in 
  Remark~\ref{rem:incomproofcritical} is supplied by 
  \eqref{divsigR}.
\end{proof}

\section{Connection to stress elements with weakly imposed symmetry}
\label{sec:weak-symmetry}

In this section we discuss connection of the methods proposed in Sections~\ref{sec:JM-element} and \ref{sec:P0-disp} to mixed finite element methods for elasticity using stresses with weakly imposed stress symmetry, i.e., the symmetry of $\sigma$ is imposed weakly by the additional variational equation 
\begin{align*}
    \int_{\Omega} \sigma : \eta = 0, \quad \eta \in \Gamma = L^2(\Omega, \mathbb{K})
\end{align*}
where $\mathbb{K}$ is the space of
$N \times N$ skew-symmetric matrices, with $N=2$ or $3$.
Introducing $\rho = \grad u - \epsilon(u)$,  the skew-symmetric part of $\grad u$,  the alternative mixed formulation 
finds $\sigma \in \SW$,  $u \in V$, $\rho\in \Gamma$ such that
\begin{subequations}\label{Elas-ws}
\begin{alignat}{2}
(\bar{\Aa} \sigma, \tau)+(u, \dive \tau)+(\rho, \tau)& =  0 \qquad && {\text{ for all }} \tau \in \SW, \label{Elas1-ws}\\
(\dive \sigma, v)& =  (f,v) \qquad  && {\text{ for all }} v \in V, \label{Elas2-ws}\\
(\sigma, \eta)& =  0\qquad && {\text{ for all }} \eta \in \Gamma, \label{Elas3-ws}
\end{alignat}
\end{subequations}
where $\SW:=H(\dive, \Omega, \mathbb{M})$
and $\bar{\Aa}$ is an operator which identical to $\Aa$ for $L^2(\Omega, \mathbb{S})$ and is bounded and coercive on $L^2(\Omega, \mathbb{K})$.
 Since $\SW_h\subset \SW$ can be a matrix-valued finite element space such that each row is a standard $H(\text{div})$ finite element, it is easier to construct stable mixed methods for \eqref{Elas-ws}.

Finite element subspaces  $\SW_{h,k} \times \VW_{h,k}\times \Gamma_{h,k} \subset \SW \times V \times \Gamma$
giving stable mixed methods for~\eqref{Elas-ws} of the form
\begin{subequations}\label{Elasp-ws-discrete}
\begin{alignat}{2}
    (\bar{\Aa} \ssigma_h^{w}, \tau)+(\uu_h^{w}, \dive \tau)+(\hat{\rho}_h, \tau)& =  0 \qquad && {\text{ for all }} \tau \in \SW_{h,k}, \label{Elasp1-ws-discrete}\\
    (\dive \ssigma_h^{w}, v)& =  (f,v) \qquad  && {\text{ for all }} v \in \VW_{h,k}, \label{Elasp2-ws-discrete}\\
    (\ssigma_h^{w}, \eta)& =  0\qquad && {\text{ for all }} \eta \in \Gamma_{h,k}, \label{Elasp3-ws-discrete} 
\end{alignat}
\end{subequations}
has long been studied~\cite{arnoldbrezzi1984peers,stenberg1988elasticity,morley1989elasticity,farhloul1997elasticity,arnold2007mixed,boffietal2009elasticity,gopalakrishnan2012second}.
In \cite{gopalakrishnan2012second} it was proved that on  meshes where each element $T$ has  a  Clough--Tocher or Alfeld split, the spaces
\begin{subequations} \label{weak-symmetry-element}
  \begin{align}
    % \label{weak-symmetry-element1}
    % \SW_h(T)&=\{\omega \in H(\dive, T; \mathbb{M}): \omega|_K \in \pol_k(K, \mathbb{M}), {\text{ for all }} K \in \Ta\}, \\
    % \label{weak-symmetry-element2}
    % W_h(T)&= \{ v \in [L^2(T)]^3: v|_K \in [\pol_{k-1}(K)]^d, {\text{ for all }} K \in \Ta\}, \\
    % \label{weak-symmetry-element3}
    % \Gamma_h(T)&= \{ \eta \in L^2(T, \mathbb{K}): \eta|_K \in \pol_{k}(K,\mathbb{K}), {\text{ for all }} K \in \Ta \}
    \label{weak-symmetry-element1}
    \SW_{h,k}&=\{\omega \in H(\dive, \om; \mathbb{M}): \omega|_K \in \pol_k(K, \mathbb{M}), {\text{ for all }} K \in \Ta, T \in \Th\}, \\
    \label{weak-symmetry-element2}
    \VW_{h,k} & = \{ v \in [L^2(T)]^N: v|_K \in [\pol_{k-1}(K)]^N, {\text{ for all }} K \in \Ta,  T \in \Th \}, \\
    % W_h(T)&= \{ v \in [L^2(T)]^3: v|_K \in [\pol_{k-1}(K)]^d, {\text{ for all }} K \in \Ta\}, \\
    \label{weak-symmetry-element3}
    \Gamma_{h,k}&= \{ \eta \in L^2(T, \mathbb{K}): \eta|_K \in \pol_{k}(K,\mathbb{K}), {\text{ for all }} K \in \Ta, T \in \Th \}
  \end{align}
\end{subequations}
give stable methods
for $k\ge N-1$ in the $N$-dimensional elasticity formulation for $N=2,3$.
Note that the polynomial degrees of $\SW_{h,k}$ and $\Gamma_{h,k}$ are the same in~\eqref{weak-symmetry-element}. Hence, the numerical solution $\ssigma_h^{w}$ of \eqref{weak-symmetry-element} is exactly symmetric pointwise. In \cite[Subsection~4.2.3]{lee2016weaksymmetry}, this result was extended to $k\ge 1$ for $N=3$.
% and the extended result inspired us to develop the three dimensional Johnson--Mercier element.
In particular, it was proved in  \cite{lee2016weaksymmetry} that
\begin{equation}
  \label{eq:10}
  \parbox{0.82\textwidth}{
    given any $w \in W_h$ and  $\eta\in \Gamma_{h,k},$
    there exists $\tau \in \SW_{h,k}$ satisfying $\dive \tau = w$,
    $(\tau, \eta) = \| \eta \|_{L^2(\Omega)}^2$,    and
    $C\| \tau \|_{H(\div,\Omega)} \le
    \| w \|_{L^2(\Omega)} + \| \eta \|_{L^2(\Omega)}$.}
\end{equation}
% \begin{itemize}
%     \item[(A)] there exists $C>0$ independent of mesh sizes such that for any given $(v, \eta) \in W_h \times \Gamma_h$ there exists $\tau \in \SW_{h,k}$ satisfying $\dive \tau = v$, $(\tau, \eta) = \| \eta \|_{L^2(\Omega)}^2$, and $\| \tau \|_{H(\div,\Omega)} \le C (\| v \|_{L^2(\Omega)} + \| \eta \|_{L^2(\Omega)} )$.
% \end{itemize}
The inf-sup stability of~\eqref{weak-symmetry-element} follows from~\eqref{eq:10}.

\begin{proposition}
  \label{prop:ws-1}
  The first two components of the unique solution
  $(\ssigma_h^{w}, \uu_h^{w}, \hat{\rho}_h)$
  of~\eqref{Elasp-ws-discrete} obtained
  using the spaces in~\eqref{weak-symmetry-element} with 
  $k=1$, coincide with the solution $(\ssigma_h,
  \uu_h)$ of~\eqref{FEMp}.
\end{proposition}
\begin{proof}
  Note that when $k=1$, the space $\VW_{h,1}$ in~\eqref{weak-symmetry-element2} equals $W_h$. By \eqref{Elasp3-ws-discrete}, $\ssigma_h^{w}$ is symmetric, so it
  is in $\Sigma_h$, the JKM space. By
  restricting $\tau_h$ to $\Sigma_h$, \eqref{Elasp1-ws-discrete} and
  \eqref{Elasp2-ws-discrete} are exactly same as \eqref{FEMp} because
  $\bar{\Aa}=\Aa$ on $H(\dive, \Omega, \mathbb{S})$. Therefore,
  $\ssigma_h^{w}=\ssigma_h$ and $\uu_h^{w} =\uu_h$ by uniqueness of
  solutions.
\end{proof}

Next, consider the case where the displacement is discretized in the space of piecewise linear functions on the unsplit mesh $\Th$, namely the space $V_h$ from Section~\ref{sec:JM-element}, i.e., 
\begin{subequations}\label{Elas-ws-discrete}
\begin{alignat}{2}
    (\bar{\Aa} \sigma_h^{w}, \tau)+(u_h^{w}, \dive \tau)+(\rho_h, \tau)& =  0 \qquad && {\text{ for all }} \tau \in \SW_{h,1}, \label{Elas1-ws-discrete}\\
    (\dive \sigma_h^{w}, v)& =  (f,v) \qquad  && {\text{ for all }} v \in V_h, \label{Elas2-ws-discrete}\\
    (\sigma_h^{w}, \eta)& =  0\qquad && {\text{ for all }} \eta \in \Gamma_{h,1}. \label{Elas3-ws-discrete} 
\end{alignat}
\end{subequations}

\begin{proposition}
  The method~\eqref{Elas-ws-discrete} is inf-sup stable and the first
  two components of its unique solution
  $(\sigma_h^w, u_h^w, \rho_h) \in \SW_{h,1} \times V_h \times \Gamma_{h,1}$
  coincide with $(\sigma_h^{w}, u_h^{w} )$ solving \eqref{FEM}.
\end{proposition}
\begin{proof}
  Given any $(v, \eta)\in V_h \times \Gamma_{h,1}$, applying~\eqref{eq:10}
  to $(Pv, \eta) \in W_h \times \Gamma_{h,1}$, we find that
  there exists $\tau \in \SW_{h,1}$ satisfying
  \begin{equation}
    \label{eq:10-B}
    \dive \tau = Pv, \quad
    (\tau, \eta) = \| \eta \|_{L^2(\Omega)}^2, \quad 
    C\| \tau \|_{H(\div,\Omega)} \le
    \| P v\|_{L^2(\Omega)} + \| \eta \|_{L^2(\Omega)}.
  \end{equation}
  Using this $\tau$, 
  \begin{align*}
    \sup_{\og \in \SW_{h,1}}
    \frac{(v, \dive \og)+(\eta, \og)}{\| \og\|_{H(\divg, \om)} }
    & \ge
      \frac{(v, \dive \tau)+(\eta, \tau)}{\| \tau\|_{H(\divg, \om)} }
      \ge C ( \| Pv \|^2_{L^2(\om)} + \| \eta \|_{L^2(\om)}^2)^{1/2}
    \\
    & \ge C ( \| v \|^2_{L^2(\om)} + \| \eta \|^2_{L^2(\om)})^{1/2},
  \end{align*}
  where we have used \eqref{enorms1}. Thus, \eqref{Elas-ws-discrete}
  is inf-sup stable and uniquely solvable.  Now we can use the same
  uniqueness argument as in the proof of Proposition~\ref{prop:ws-1}
  to show that $(\sigma_h^{w}, u_h^{w} )$ and the solution of
  \eqref{FEM} are equal.
\end{proof}

% %$(\sigma_h^{w}, u_h^{w})$ of its numerical solution is same as the ones of \eqref{FEM}. 
% For stability of this system, note that a direct application of (A) gives
% \begin{itemize}
%     \item[(B)] there exists $C>0$ independent of mesh sizes such that for any given $(v, \eta) \in V_h \times \Gamma_h$ there exists $\tau \in \SW_h$ satisfying $\dive \tau = Pv$, $(\tau, \eta) = \| \eta \|_{L^2(\Omega)}^2$, and $(\| \tau \|_{L^2(\Omega)}^2 + \|\div \tau\|_{L^2(\Omega)}^2)^{1/2} \le C (\| Pv \|_{L^2(\Omega)} + \| \eta \|_{L^2(\Omega)} )$.
% \end{itemize}
% The inf-sup stability for \eqref{Elas-ws-discrete} then follows by the combination of (B) and \eqref{enorms1}. Then, we can show that $(\sigma_h^{w}, u_h^{w} )$ and the solution of \eqref{FEM} are equal by the argument we used for $\SW_h \times W_h \times \Gamma_h$. 

To conclude this section, we have shown that it is possible to compute
the solutions of methods in Sections~\ref{sec:JM-element}
and~\ref{sec:P0-disp} using completely standard finite element spaces
via  formulations~\eqref{Elasp-ws-discrete}--\eqref{weak-symmetry-element}
and~\eqref{Elas-ws-discrete}.
However, we do not know how to arrive at
the reduced space $\Sigma_h^R \times V_h^R$ of
Section~\ref{sec:reduced-element} through methods with weakly imposed stress symmetry using standard spaces.
The space $\Sigma_h^R(T)$ does not have full
$\pol_1(K,\mathbb{M})$ polynomials for $K \in \Ta$, so finding an
appropriate $\Gamma_h^R \subset \Gamma$ which gives exact symmetry of
stress tensors and satisfies the inf-sup stability, appears to be nontrivial.

\section{The $N$-dimensional case}
\label{sec:ndim-case}

This section generalizes the JKM element to higher
dimensions.  To generalize our arguments in
Section~\ref{sec:JM-element}, we need an $N$-dimensional version of
the identity~\eqref{alg1-1}. This can be found, described in the
language of exterior calculus, in~\cite{arnold2006finite,
  arnold2021complexes, Falk08}, using the so-called ``BGG resolution''~\cite{arnold2007mixed,calabi1961compact, eastwood2000complex}. Nonetheless, we start by giving an
elementary and self-contained proof of such an identity using standard
vector notation in Subsection~\ref{ssec:comm-diagr-tens}.  We then use
it to develop the JKM element in $\mathbb{R}^N$ for any
$N\ge 2$.

\subsection{A commuting diagram of tensor fields for
  $N$-dimensional elasticity}  \label{ssec:comm-diagr-tens}

Let $\V = \R^N$, let $\M$ denote the space of $N \times N$ matrices, and let $\K = \skw(\M)$ denote the space of skew symmetric matrices. 
We use  $\otimes$ to denote the tensor product of vector spaces. Then
for (column) vectors  $w, v \in \V$, 
$w \otimes v= w v^t$ equals the outer product. Let 
$e_i$ denote the standard unit basis of $\V$,
$\EK_{ij}=2 \skw(e_i \otimes e_j)= e_i e_j^t- e_j e_i^t \in \K$.
% Note that $\V \otimes \V$ is isomorphic to $\M$ and we use the notation $w \otimes v= v w^t$ where $w, v \in \V$ are considered column vectors. 
Let
$\Lambda(\X)$ be $C^\infty$-smooth $\X$-valued fields where $\X$ is
one of $\V, \M,$ $\K$ or tensor products of these. 
Henceforth, we use Einstein's summation convention.

We continue to use the standard divergence operator on vector fields
in $\Lambda(\V)$ and the (previously used) row-wise divergence
operator on matrix fields in $\Lambda(\M)$. The \revj{ $\divg$ operator is also defined for elements in}   $\K \otimes \Lambda(\V)$ by 
\begin{equation*}
  \dive( \eta\otimes w )= \eta \otimes (\dive w),\, \qquad
  \eta \in \K, \;
  w \in \Lambda(\V).
\end{equation*}
This motivates the \revj{ definition of divergence acting on} the isomorphic
space $\Lambda( \K \otimes \V)$: writing an element of
$\Lambda( \K \otimes \V)$ as $b = b_{ijk} \EK_{ij} \otimes e_k$ for
some scalar component functions $b_{ijk}(x)$, we define
$\divg : \Lambda( \K \otimes \V) \to \Lambda(\K)$ by
\begin{equation}
  \label{eq:div-KV}
  \divg (b_{ijk} \EK_{ij} \otimes e_k) = (\partial_k b_{ijk}) \EK_{ij}.
\end{equation}
Next, define a differential operator
$\dd: \Lambda(\K) \rightarrow   \Lambda(\V)$ by
\begin{equation}
  \label{eqdd-1}
\dd(\eta_{ij} \EK_{ij}):= \partial_{j}
(\eta_{ij} -\eta_{ji}) e_i, \qquad \eta = \eta_{ij} \EK_{ij} \in \Lambda(\K).
\end{equation}
It is natural to extend $\dd$ to $\V \otimes \Lambda(\K)$ by
\begin{equation*}
  \dd(v \otimes \omega )=  v \otimes \dd \omega, \qquad v \in \V,
  \omega \in \Lambda(\K).
\end{equation*}
Since $\Lambda( \V \otimes \K)$ is isomorphic to
$\V \otimes \Lambda(\K)$, this motivates the definition of the
extension of $\dd$ in \eqref{eqdd-1} to
$\dd : \Lambda( \V \otimes \K) \to \Lambda(\M)$ by
\begin{equation}
  \label{eq:dd-2}
  \dd ( a_{ijk} e_i \otimes \EK_{jk}) = \partial_k( a_{ijk} - a_{ikj}) \,
  e_i \otimes e_j,
  \qquad  a_{ijk} e_i \otimes \EK_{jk}\in \Lambda(\V \otimes \K).  
\end{equation}

\begin{remark}
  Alternately, we can arrive at these definitions through a generalized
  divergence operator~\cite{MarsdHughe94}, $\Div$, that acts on
  $\Lambda(\V \otimes \V)$ and $\Lambda (\V \otimes \V \otimes \V)$ by
  \begin{gather}
    \label{eq:Div-def}
    % \Div (a_i e_i) = \partial_i a_i,
    % \qquad
    \Div (a_{ij} e_i \otimes e_j) = (\partial_j a_{ij} ) e_i,
    \qquad
    \Div (a_{ijk} e_i \otimes e_j \otimes e_k) =
    (\partial_k a_{ijk} ) e_i\otimes e_j.
  \end{gather}
  Since $\K \otimes \V$ and $\V \otimes \K$ are both subspaces of
  $\V \otimes \V \otimes \V$, the latter defines $\Div$ on
  $\Lambda(\K \otimes \V)$ and $\Lambda(\V \otimes \K)$.  Since 
  \begin{align*}
    \eta
    & = \eta_{ij} \EK_{ij} = (\eta_{ij} - \eta_{ji}) e_i \otimes e_j
      && \in \Lambda(\K),
    \\
    a & = a_{ijk} e_i \otimes \EK_{jk} = (a_{ijk} - a_{ikj}) e_i \otimes
       e_j \otimes e_k && \in \Lambda( \V \otimes \K),
    \\
    b &  = b_{ijk} \EK_{ij} \otimes e_k = (b_{ijk} - b_{jik} ) e_i \otimes
        e_j \otimes e_k && \in \Lambda( \K \otimes \V),   
  \end{align*}
  applying~\eqref{eq:Div-def}  we easily see that
  equations~\eqref{eq:div-KV}, \eqref{eqdd-1}, and~\eqref{eq:dd-2} are
  the same as
  \[
    \divg b = \Div b, \quad 
    \dd \eta = \Div \eta, \quad 
    \dd a = \Div a,
  \]
  respectively.   
\end{remark}

With these definitions,  we have
\begin{align}
  \label{auxdddive-1}
 \dive  \dd\, \eta & =0,  && {\text{ for all }} \eta \in \Lambda(\K),  
  \\
  \label{applemma1-1}
  \dive \dd \, \omega & =  0,
  && {\text{ for all }} \omega \in \Lambda(\V \otimes \K).
\end{align}
Indeed, to see that \eqref{auxdddive-1} holds, 
it suffices to observe that
$\dive \dd \eta = \dive (\partial_{j} (\eta_{ij} -\eta_{ji}) e_i)=
\partial_{i}\partial_{j} (\eta_{ij} -\eta_{ji}) =0. $ Essentially the
same argument also shows~\eqref{applemma1-1}.

\begin{remark}
In the language of exterior calculus, \eqref{auxdddive-1} and
\eqref{applemma1-1} imply that the sequences
\begin{equation}
  \label{eq:cmplx1}
      \begin{tikzcd}
        & \Lambda(\K)
        \arrow{r}{\dd}
        &
        \Lambda(\V)
        \arrow{r}{\dive}
        & \Lambda(\R),
      \end{tikzcd}
\end{equation}
\begin{equation}
  \label{eq:cmplx2}
      \begin{tikzcd}
        & \Lambda(\V \otimes \K)
        \arrow{r}{\dd}
        &
        \Lambda(\M)
        \arrow{r}{\dive}
        & \Lambda(\V),
      \end{tikzcd}
\end{equation}
form complexes, elucidating the connection to relevant $N$-dimensional
de Rham complexes. Note that the tensor products of $\V$ with function
spaces in~\eqref{eq:cmplx1} yield their analogue in~\eqref{eq:cmplx2}
up to an isomorphism.  
\end{remark}

Next, we define a linear operator
$\Theta: \Lambda(\V \otimes \K) \rightarrow \Lambda(\K \otimes \V)$ by defining
$\Theta a$ for $a = a_{ijk} e_i \otimes e_j \otimes e_k$ in the larger
set $\Lambda(\V \otimes \V \otimes \V)$, namely 
\begin{equation}
  \label{eq:Theta-def}
  \Theta (a_{ijk} e_i \otimes e_j \otimes e_k)  := (a_{ijk} - a_{jik}) \; e_i \otimes e_j \otimes e_k  % = a_{ijk} \EK_{ij} \otimes e_k.  
\end{equation}
When $a$ is in $\Lambda(\V \otimes \K)$, we have 
$a = a_{ijk} e_i \otimes e_j \otimes e_k = \frac 1 2 a_{ijk} e_i
\otimes \EK_{jk}$. Hence, \eqref{eq:Theta-def} can equivalently be expressed as 
\begin{equation}
  \label{eq:Theta-def-2}
\Theta (a_{ijk} e_i \otimes \EK_{jk}) = 2 a_{ijk} \EK_{ij} \otimes
e_k.  
\end{equation}

\begin{lemma}
  The operator
  $ \Theta: \Lambda(\V \otimes \K) \rightarrow \Lambda(\K \otimes \V)$
  is invertible and  its inverse is given by
  \[
    \Theta^{-1}\, b
    = \frac 1 2 (b_{ijk} - b_{ikj} - b_{jki}) \,e_i \otimes e_j \otimes e_k
  \]
  for any $b = b_{ijk} \, e_i \otimes e_j \otimes e_k\in \Lambda(\K \otimes \V)$.
  Moreover, for all $\omega \in \Lambda(\V \otimes \K)$, 
  \begin{align}
    \dive \Theta \,\omega 
    & = 2 \skw \dd \,\omega. 
      \label{applemma2-1}       
  \end{align}
\end{lemma}
\begin{proof}
  By direct calculation, it is easy to verify that the given
  expression for $\Theta^{-1}$ satisfies $ \Theta (\Theta^{-1} b) = b$
  for all $b \in \Lambda(\K \otimes \V)$ and $ \Theta^{-1} (\Theta a) = a$ for all
  $a \in \Lambda(\V \otimes \K)$. To detail the latter, letting
  $\theta = \Theta a = (a_{ijk} - a_{jik}) e_i \otimes e_j \otimes
  e_k,$ observe that
  \begin{align*}
    2 \Theta^{-1} \Theta a
    &= 
      (\theta_{ijk} - \theta_{ikj} - \theta_{jki}) e_i \otimes e_j \otimes   e_k
    \\
    & =
      ( (a_{ijk} - a_{jik}) - (a_{ikj} - a_{kij}) - (a_{jki} - a_{kji})
      ) e_i \otimes e_j \otimes  e_k
  \end{align*}
  When $a \in \Lambda(\V \otimes \K)$, its components have skew
  symmetry in the last two indices (i.e., $a_{ijk} = -a_{ikj}$), so
  the last expression above simplifies to
  $2a_{ijk} e_i \otimes e_j \otimes e_k = 2a$.

  Next, to prove~\eqref{applemma2-1}, let
  $\og = \omega_{ijk} e_i \otimes \EK_{jk}\in \Lambda( \V \otimes
  \K)$. Then, by~\eqref{eq:div-KV} and \eqref{eq:Theta-def-2}, 
  \begin{align*}
    \div (\Theta \omega)
    &  % = \div( (\omega_{ijk} - \omega_{jik}) e_i \otimes e_j \otimes e_k)
      = \div( 2\omega_{ijk} \EK_{ij} \otimes e_k)
     = 2(\partial_k \omega_{ijk})\, \EK_{ij}.
  \end{align*}
  Also, by~\eqref{eq:dd-2}, 
  \begin{align*}
    2\skw  \dd \omega
    & 
      % \skw  \dd (  \omega_{ijk} e_i \otimes \EK_{jk})
      =  \partial_k (\omega_{ijk} - \omega_{ikj}) \,2 \skw(e_i \otimes e_j)
    \\
    &  = \partial_k (\omega_{ijk} - \omega_{jik}) \; \EK_{ij} 
      = 2(\partial_k \omega_{ijk})\, \EK_{ij}.
  \end{align*}
  Thus the left and right hand sides of~\eqref{applemma2-1} are equal.
\end{proof}

In other words,  we have shown that the following diagram commutes and that its  top and bottom rows are complexes.
\begin{equation*}
      \begin{tikzcd}
        &
        &
        \Lambda(\K \otimes \V)
        \arrow{r}{\dive}
        &  \Lambda(\K) 
        \\
        & \Lambda(\V \otimes \K)
        \arrow{r}{\dd}\arrow[ur, "\Theta", sloped]
        &
        \Lambda(\M)
        \arrow{r}{\dive}\arrow[ur, "2\mathrm{skw}", sloped]
        & \Lambda(\V) .
      \end{tikzcd} 
    \end{equation*}

% \begin{remark}[Specialization to 3D] 
%   When $N=3$, it can be easily verified that the above definitions
%   imply the following identities:
%   \begin{align*}
%     \dd \mskw & =-  \curl \\
%     \Theta (\mskw_L ( v\otimes w))
%     & = -\mskw_R \Xi (v \otimes w), \quad  v, w \in \V,
%   \end{align*}
%   where $\mskw_L(\psi \otimes \phi)= \mskw (\psi) \otimes \phi$ and
%   $\mskw_R(\psi \otimes \phi)=\psi \otimes \mskw(\phi)$ for
%   $\psi, \phi \in \V$.  Thus $\Theta$ generalizes the
%   three-dimensional operator $\Xi$.  The analogue of the 3D identity
%   \eqref{alg1-1} in $N$~dimensions is \eqref{applemma2-1}. 
% \end{remark}

For later use, we state the following result on composition of  trace with $\dd$ on a hyperplane.
\begin{lemma}\label{traceGamma}
Let $\eta= \eta_{ij} \EK_{ij} \in \Lambda(\K)$ and suppose that $\Gamma$ is a hyperplane in which the components of $\eta$ vanish then $\dd \eta  \cdot n$ vanishes on $\Gamma$ where $n$ is a unit normal to $\Gamma$.   
\end{lemma}
\begin{proof}
After possibly rotating coordinates we can assume without loss of generality that $n=e_N$. Then, 
\begin{equation*}
\dd \eta \cdot n=\partial_j (\eta_{ij}-\eta_{ji}) e_i \cdot e_N= \partial_j (\eta_{Nj}-\eta_{jN})=\sum_{j=1}^{N-1} \partial_j (\eta_{Nj}-\eta_{jN}). 
\end{equation*}
The latter vanish on $\Gamma$ since $\eta_{Nj}-\eta_{jN}$ vanish on $\Gamma$  and $\partial_j$ (for $1 \le j\le N-1$) are tangential derivatives on $\Gamma$. 
\end{proof}

Of course, we can define the weak version of $\dd$ and extend the domain of $\dd$. To this end, let $\eta =\eta_{ij} \EK_{ij}$ we say $\dd \eta$ exists if there exists $\omega \in  L^1(\Omega, \V)$ such that
\begin{equation}\label{weakdd}
( \omega ,\phi_k  e_k)=  -( \eta_{ij} -\eta_{ji} , \partial_j \phi_k) \delta_{ik}, \qquad {\text{ for all }}  \phi = \phi_k e_k  \in C_0^\infty(\Omega, \V),   
\end{equation}
and we set $\dd \eta= \omega$. Here $\delta_{ik}$ is the Kronecker delta. We then define $H(\dd, \Omega):= \{ \eta \in L^2(\Omega, \K): \dd \eta \in L^2(\Omega, \V)\}$. The  identities \eqref{applemma1-1} and \eqref{applemma2-1} hold for $\omega \in \V \otimes H(\dd, \Omega)$.

\subsection{The stress element on the Alfeld split in higher dimensions}
\label{ssec:stress-elem-ndim}

For an $N$-simplex $T=[x_0, x_1, x_2, \cdots, x_N]$, let $\Ta$ be an
Alfeld split of $T$, i.e., choosing an interior point $z$ of $T$,
define $\Ta= \{ T_0, T_1, T_2, \cdots, T_N\}$ with $N$-simplices
$T_i=[z, x_0, \cdots, \widehat{x_i}, \cdots, x_N]$ where
$\widehat{x_i}$ means that $x_i$ is not in the vertices of $T_i$.  For
a given triangulation $\Th$ of $\Omega$, we let $\Tha$ be the
resulting triangulation after performing an Alfeld split to each
$T \in \Th$. On each macro element $T \in \Th$, define the  local spaces
\begin{alignat*}{1}
\pol_k(\Ta, \X):=&\{ \omega \in L^2(T, \X): \omega|_K \in \pol_k(K, \X), {\text{ for all }} K \in \Ta \},  \\
\Sigma_h(T):=& H(\div, T, \mathbb{S}) \cap \pol_1(\Ta, \mathbb{S}), \\
%\{ \omega \in H(\div, T, \mathbb{S}): \omega|_K \in \pol_1(K, \mathbb{S}), \text{ for all } K \in \Ta\}, \\
V_h(T):=&  \pol_1(T, \mathbb{V}). 
%W_h(T):=&\{ v \in L^2(T, \V): v \in \pol_0(K, \mathbb{V}),
%\text{ for all }  K \in \Ta\}.
\end{alignat*}
%Note that $\dim W_h(T) = N(N+1) = \dim V_h(T)$, so $W_h(T)$ and $V_h(T)$ have the same dimension.
The $N$-dimensional  matrix analogue of the well-known BDM space
and its subspace with vanishing normal components on $\partial T$ are 
  namely
  \begin{align*}
    \BDM_1(\Ta)
    &:=  H( \div, T, \M) \cap \pol_1(\Ta, \M), \\
    %   \{ \tau \in H(\div, T, \M): \tau|_K \in \pol_1(K, \M), \text{ for all }
  %               K \in \Ta \} ,\\
    \BDMo_1(\Ta)
    &:= \{ \tau \in \BDM_1(\Ta): \tau \bn |_{\partial T} =0 \} .
  \end{align*}
  Notice that the Alfeld split of $T$ has $N+1$ distinct sub-simplices in $\triangle_{N-1}(T)$ and $\frac 12 N(N+1)$ distinct internal subsimplices in $\triangle_{N-1}(\Ta)$, therefore $\dim \BDM_1(\Ta) = (N+1)\cdot N^2 + \frac 12 N(N+1)\cdot N^2=(\frac N 2+1)(N+1)N^2$ and $\dim \BDMo_1(\Ta) = \frac 12 (N+1)N^3$. Notice also that $\dim \pol_1(\Ta, \mathbb{K}) = \frac 12 N(N-1)(N+1)^2$.

  We will also need the space  $\LLo_2(\Ta):= \pol_2(\Ta) \cap \mathring{H}^1(T)$ and $\X$-valued versions $\LLo_2(\Ta, \X)= \LLo_2(\Ta) \otimes \X$.
We begin by recalling   the next result  from  \cite[Theorem~3.1]{fu2018exact} (see also \cite{guzman2018inf}).
 \begin{proposition}\label{divonto}
Let $v \in \pol_1(T^A)$ with $\int_T v=0$ there exists $\rho \in \LLo_2(\Ta, \V)$ such that 
\begin{equation*}
\dive \rho=v.
  \end{equation*}
\end{proposition}

\begin{proposition}\label{ddBDM}
If  $\rho \in \LL_2(\Ta,  \V \otimes \K)$ then $\dd \rho \in  \BDM_1(\Ta)$. Moreover,  if $\rho \in \LLo_2(\Ta, \V \otimes  \K)$ then $\dd \rho \in  \BDMo_1(\Ta)$. 
\end{proposition}
\begin{proof}
Let $\rho \in  \LL_2(\Ta,  \V \otimes \K)$ then for  each $K \in \Ta$  $\dd \rho|_K$ has linear components. Hence, it is enough to show that $\rho \in H(\dd, T) \otimes \V$, but this follows immediately from \eqref{weakdd}.  If  $\rho \in \LLo_2(\Ta,  \V \otimes \K)$ then we must also show that $(\dd \rho) \, n$ vanishes on $\partial T$, but this follows from Lemma~\ref{traceGamma}. 
\end{proof}

\begin{lemma}\label{SigmazeroN}
  The equality 
  $\{ \omega \in \Sigma_h(T): \dive \omega=0, \omega n|_{\partial T}=0\}=\{0\}$
  continues to hold in the $N$-dimensional case. 
\end{lemma}
\begin{proof}

Let $\tilde \pol_0(\Ta) = \{ u \in \pol_0(\Ta): \int_T u = 0\}$ and
  set
  \[
    X = \left\{ (\eta, v):  \eta \in \pol_1(\Ta, \mathbb{K}), \;
      v \in \tilde \pol_0(\Ta, \V), \;
      \int_T (\eta : \kappa + v \cdot \kappa x ) = 0 \text{ for all } \kappa \in \mathbb{K}
    \right\}
  \]
  where $\kappa x$ is the matrix-vector product with the coordinate vector polynomial $x$. 
  The proof is based on  the operator $A : \BDMo_1(\Ta) \to X$ given by
  \begin{equation}
    \label{eq:Adefn-ndim'}
    A \tau = (-\skw \tau, \dive \tau),\qquad \tau \in \BDMo_1(\Ta).
  \end{equation}
  Observe that for any $\kappa \in \mathbb{K}$ and $\tau \in \BDMo_1(\Ta)$,
  \begin{equation}
    \label{eq:5-ndim'}
    \int_T \kappa : \skw \tau = \int_T \grad (\kappa x):\tau 
    = -\int_T (\kappa x) \cdot \div \tau .    
  \end{equation}
  Thus $A$ indeed maps into $X$. We proceed to  show that $A$ is surjective.

  Let $(\eta, v) \in X$. Since $v$ has components of zero
  mean, by a standard exact sequence property~\cite{arnold2018finite},
  there exists a $\sigma \in \BDMo_1(\Ta)$ such that
  \begin{equation}
    \label{eq:divsig-ndim'}
    \dive \sigma = v.
  \end{equation}
  Since $(\eta, v ) \in X$, this implies 
  \[
    \int_T (\eta - \skw \sigma) : \kappa = \int_T (\eta : \kappa
    + (\div \sigma ) \cdot \kappa x) 
    = \int_T (\eta : \kappa   + v \cdot \kappa x)  = 0,
  \]
  where we have again used~\eqref{eq:5-ndim'}.  Letting
  $\zeta = \eta - \skw \sigma \in \pol_1(\Ta, \K),$ the above shows all the components have zero
  mean.  Hence, by Proposition \ref{divonto} there exists $\omega \in  \LLo_2(\Ta, \V \otimes \K)$ such that 
\begin{equation}\label{eq:diveta-ndim'}
    \dive  \omega = \eta - \skw \sigma.
\end{equation}
  We set $a= \Theta^{-1} \omega$ and note that $a \in  \LLo_2(\Ta,  \V \otimes \K)$. Setting $\tau = \sigma + \dd \, a$, then by Proposition \ref{ddBDM} we have $\tau \in \BDMo_1(\Ta)$.

  By the commuting diagram property~\eqref{applemma2-1}, 
  $\skw \dd \, a = \div \Theta a = \div \omega$, so we have
  \begin{align*}
    \skw \tau & = \skw \sigma + \div \omega = \eta ,
    && \text{(by~\eqref{eq:diveta-ndim'})}
    \\
    \dive \tau & = v.
    && \text{(by~\eqref{eq:divsig-ndim'} and \eqref{applemma1-1})}
  \end{align*}
  Thus we have proved that $A$ is surjective.

  To conclude, the surjectivity of $A$ implies that 
  \begin{align*}
    \rank(A) &\ge \dim \pol_1(\Ta, \mathbb{K}) + \dim \tilde \pol_0(\Ta,\V) - \dim \mathbb{K}
    \\
             &= \frac 12 N(N-1)(N+1)^2 + N^2 - \frac 12 N(N-1)
    \\
             &= \frac 12 N^3(N+1). 
  \end{align*}
  Since $\dim \BDMo_1(\Ta) = \frac 12 N^3(N+1)$, the rank-nullity theorem implies that the null space of $A$ is trivial,
  i.e.,
  $\{ \tau \in \BDMo_1(\Ta) : \; \skw \tau =0, \; \dive \tau=0\} =
  \{ 0\}$. 
  \end{proof}

We now prove the main result of this section. The proof is  a straightforward generalization of the proof of Theorem~\ref{thmdofs1}.
\begin{theorem}\label{thm:n-dim-Sigmazero}
  The dimension of $\Sigma_h(T)$ is $(N+\frac 12)N(N+1)$. Moreover,
  an element $\omega \in \Sigma_h(T)$ is uniquely determined by the
  following dofs:
\begin{subequations}
\label{Sigmah-ndim'}
\begin{alignat}{4}
  &\int_F \omega \bn_T \cdot \kappa \, {\ds}, \qquad
  &&\kappa\in \pol_1(F, \V),
  \quad F \in \triangle_{N-1}(T), \qquad
  && \dofcnt{($(N+1)N^2$ dofs)} \label{Sigmah1-ndim'}
  \\
  &\int_T \omega.
  &&\quad && \dofcnt{($\frac 12 N(N+1)$ dofs)} \label{Sigmah2-ndim'}
\end{alignat}
\end{subequations}

\end{theorem}
\begin{proof}
Since $\Sigma_h(T)=\{ \omega \in \BDM_1(\Ta): \int_T \omega : \eta = 0 \; \text{ for all } \eta \in \pol_1(\Ta, \mathbb{K})\}$, 
  \begin{align*}
    \dim \Sigma_h(T) &\ge \dim \BDM_1(\Ta) - \dim \pol_1(\Ta, \mathbb{K})
    \\
    & = (\frac N2 +1)(N+1)N^2 - \frac N2 (N-1)(N+1)^2 \\
                     &= \frac N2 (N+1)(2N+1).
  \end{align*}
  To show $\dim \Sigma_h(T) = \frac N2 (N+1)(2N+1)$, it is sufficient to prove that $\omega = 0$ if the dofs \eqref{Sigmah-ndim'} of $\omega \in \Sigma_h(T)$ vanish because the number of dofs in \eqref{Sigmah-ndim'} is $\frac N2 (N+1)(2N+1)$. This also proves that $\Sigma_h(T)$ is unisolvent by \eqref{Sigmah-ndim'}. 

 To this end, let  $\omega \in \Sigma_h(T)$ have vanishing dofs \eqref{Sigmah-ndim'}. Then,  $\omega n=0$ on $\partial T$ by \eqref{Sigmah1-ndim'}. Moreover,  using \eqref{Sigmah2-ndim'} and integration by parts gives
  \begin{align}
      \int_T \div \omega \cdot v = -\int_T \omega : \epsilon(v) = 0 \qquad  {\text{ for all }}        v \in \pol_1(T, \V).
  \end{align}
 By a simple argument using quadrature rules as in the 3D case,
  we conclude that $\div \omega = 0$. Hence, Lemma \ref{SigmazeroN} implies that $\omega=0$. 
\end{proof}

% \section{Conclusion}

\appendix

\section{Another reduced space pair}
\label{sec:2nd-reduced-element}

In this section we reduce the spaces even further from the reduced pair in Section~\ref{sec:reduced-element}. In this further reduced finite element method, $V_h$ is replaced by only piecewise constant elements on $T \in \mathcal{T}_h$ and $\Sigma_h$ is also further reduced correspondingly. 
% The stability and error estimates of these reduced finite elements will be discussed. 
We prove optimal $\mathcal{A}$-weighted $L^2$ error estimate for
$\sigma$ and $L^2$ error estimate for $u$. However, for this further
reduced method, we are not able to prove
(an analogue of Theorem~\ref{thm:ee-robust} giving)
a robust error estimate 
for the nearly incompressible case. The reason for this is that we now only
have very weak
% the method lacks proper
control of the divergence of the stresses.  Hence we are unable to
recommend this method. Nonetheless, we present it as a curiosity,
since it has only three dofs per facet and yet remains stable. % (contrary to the folklore that six dofs per facet are needed to control
% rigid displacements).

%Although we can prove estimates for the stress and displacement variables in the $L^2$-norm, 

The discrete spaces for this method are 
\begin{alignat*}{1}
\tilde{\Sigma}_h^R(T):& =  \{ \omega \in \Sigma_h(T): \dive \omega \in P_T \RM(T),  \omega n \in [\pol_0(F)]^3, {\text{ for all }} F \in \triangle_2(T)\},\\
\tilde{V}_h^R(T):& =  [\pol_0(T)]^3.
\end{alignat*}
The corresponding global spaces are denoted by  $\tilde{\Sigma}_h^R, \tilde{V}_h^R$. 

\begin{theorem}
An element $\omega \in  \tilde{\Sigma}_h^R(T)$  is uniquely determined by the following dofs:
\begin{alignat}{4}
&\int_F \omega \bn  \cdot \kappa \, {\ds}, \qquad &&\kappa\in [\pol_0(F)]^3, F \in \triangle_2(T),\qquad && \dofcnt{(12 dofs).} \label{SigmaRh1t}
\end{alignat}
\end{theorem}
\begin{proof}
We see that $\dim \tilde{\Sigma}_h^R(T) \ge \dim \Sigma_h(T)- (6+6 \cdot 4)=12$ which is exactly the number of dofs in \eqref{SigmaRh1t}. Now suppose that $\omega \in \tilde{\Sigma}_h^R(T)$ and the  dofs \eqref{SigmaRh1t} vanish. Then,  we have  $\omega \bn =0$ on $\partial T$.

Let $v \in \RM(T)$ be such that $P_T v= \dive \omega$ then
\begin{alignat*}{1}
\|\dive \omega\|_{L^2(T)}^2 = \int_T \dive \omega \cdot P_T v =\int_T \dive \omega \cdot v = \int_{\partial T} \omega n \cdot v= 0.
\end{alignat*}
 Hence, this shows that $\dive \omega=0$. By Lemma \ref{Sigmazero} we have that $\omega$ vanishes. 
\end{proof}
The corresponding projection with the degrees of freedom is defined as follows. 
\begin{subequations}
\label{PiRt}
\begin{alignat}{4}
&\int_F (\tilde{\Pi}^R_T \omega \bn) \cdot  \kappa \, {\ds}= \int_F  \omega \bn  \cdot \kappa \, {\ds} , \qquad &&\kappa\in [\pol_0(F)]^3, F \in \triangle_2(T).  \label{Pi1Rt}
\end{alignat}
\end{subequations}
From this we define the global projection $\tilde{\Pi}^R$.

\begin{lemma}
It holds, 
\begin{equation}\label{commutePiglobalRt}
 (\dive( \tilde{\Pi}^R \omega -\omega),  v) =0 \qquad {\text{ for all }} v \in \tilde{V}_h^R, \omega \in    D_\Pi. 
\end{equation}
\end{lemma}
\begin{proof}
Let $v \in [\pol_0(T)]^3$. Then, 
\begin{alignat*}{2}
(\dive( \tilde{\Pi}^R \omega -\omega), v)_T& =  \int_{\partial T } (\tilde{\Pi}^R \omega -\omega) \bn \cdot v = 0. 
\end{alignat*}
\end{proof}

We omit the proofs of the following results analogous to previous sections.

\begin{proposition}
  For all $\omega \in  H^1 (\Omega, \mathbb{S})$,
  \begin{alignat}{2}
    \| \omega-\tilde{\Pi}^R \omega\|_{L^2(\Omega)}+ h \|\dive(\omega-\tilde{\Pi}^R \omega)\|_{L^2(\Omega)} \le  & C h \|\omega\|_{H^1(\Omega)}. \label{PiboundRt}
  \end{alignat}
\end{proposition}
\begin{theorem}
\label{infsupRt}
There exists a constant $\beta>0$ such that 
\begin{equation}
   \beta \le \inf_{0 \not= v \in \tilde{V}^R_h} \sup_{0 \neq \omega \in \tilde{\Sigma}_h^R}  \frac{( \dive \omega, v)}{\|\omega\|_{H(\dive,\Omega)} \|v\|_{L^2(\Omega)} }. 
\end{equation}
\end{theorem}
% \begin{proof}
%      Let $v \in \tilde{V}_h^R$ then there exists $\tau \in H^1(\Omega, \mathbb{S})$ such that 
%   \begin{alignat}{2}
%   \dive \tau= &v  \qquad  && \text{ on } \Omega \label{tau1Rt}\\
%   \|\tau\|_{H^1(\Omega)} \le & C \| v\|_{L^2(\Omega)}. \label{tau2Rt} \qquad &&
%   \end{alignat}
%   Let $\omega= \tilde{\Pi}^R \tau$ then we have 
% \begin{equation}
% \|v\|_{L^2(\Omega)}^2= (\dive \tau, v)= (\dive \omega, v), 
% \end{equation}
% by \eqref{commutePiglobalRt}. The result now follows after using \eqref{PiboundRt} 
%  and \eqref{tau2Rt} which gives $\|\omega\|_{H(\dive, \Omega)} \le C \|\tau\|_{H^1(\Omega)} \le C \| v\|_{L^2(\Omega)}$. 
% \end{proof}

{\bf{The mixed method with the second reduced element}} finds $\tsigma^R_h \in \tilde{\Sigma}^R_h$ and $\tu^R_h \in \tilde{V}^R_h$ satisfying
\begin{subequations}\label{FEMRt}
\begin{alignat}{2}
(\Aa \tsigma^R_h, \tau)+(\tu^R_h, \dive \tau)& =  0 \qquad && {\text{ for all }} \tau \in \tilde{\Sigma}^R_h, \label{FEM1Rt}\\
(\dive \tsigma^R_h, v)& =  (f,v) \qquad  && {\text{ for all }} v \in \tilde{V}^R_h. \label{FEM2Rt}
\end{alignat}
\end{subequations}

The error analysis of this method is more involved than the previous cases since $\dive( \tilde{\Pi}^R  \sigma-\tsigma^R_h)$  is not necessarily zero. We proceed
by proving a stability result for the method, as in \cite{arnold2006finite}, using
the bilinear form
\begin{equation*}
B(\omega_h, v_h; \tau_h, w_h):= (\Aa \omega_h, \tau_h)+(v_h, \dive \tau_h)-(\dive \omega_h, w_h). 
\end{equation*}

\begin{theorem} \label{thm:Babuska-Aziz-infsup}
There exists constants $c_1$ and $c_2$ such that 
for every $\omega_h \in \tilde{\Sigma}^R_h$ and $v_h \in \tilde{V}_h^R$, there exists $\tau_h \in \tilde{\Sigma}^R_h$ and $w_h \in \tilde{V}_h^R$, satisfying
\begin{alignat}{1}
\| \omega_h\|_{\Aa}^2+ \|\tilde{P}^R \dive \omega_h\|_{L^2(\Omega)}^2+\|v_h\|_{L^2(\Omega)}^2 \le c_1  B(\omega_h, v_h; \tau_h, w_h), \label{thminf1t}
\end{alignat}
and
\begin{alignat}{1}
 &  \|\tau_h\|_{\Aa}^2+ \|\tilde{P}^R \dive \tau_h\|_{L^2(\Omega)}^2 + \|w_h\|_{L^2(\Omega)}^2  \nonumber \\
 & \le  c_2 (\| \omega_h\|_{\Aa}^2+ \|\tilde{P}^R \dive \omega_h\|_{L^2(\Omega)}^2+\|v_h\|_{L^2(\Omega)}^2). \label{thminf2t}
\end{alignat}
\end{theorem}
\begin{proof}
By the proof of Theorem \ref{infsupRt} there exists $\rho_h \in \tilde{\Sigma}^R_h$ such that
\begin{alignat}{1}
\|v_h\|_{L^2(\Omega)}^2 = & (\dive \rho_h, v_h), \label{311}\\
\|\rho_h\|_{H(\dive,\Omega)} \le & \kappa \| v_h\|_{L^2(\Omega)}. \label{312}
\end{alignat}
Choose $\tau_h=\omega_h+ \frac{1}{\gamma \kappa^2} \rho_h$ and $w_h= v_h-\tilde{P}^R \dive \omega_h$. Then, we have 

\begin{alignat*}{1}
  B(\omega_h, v_h; \tau_h, w_h)=
  &  \|\omega_h\|_{\Aa}^2+\frac{1}{\gamma \kappa^2} (\Aa \omega_h,\rho_h) + (v_h, \dive \omega_h)+ \frac{1}{ \gamma \kappa^2} (v_h, \dive \rho_h) \\
&-(\dive \omega_h, v_h)+ (\dive \omega_h, \tilde{P}^R \dive \omega_h )    \\
& =   \| \omega_h\|_{\Aa}^2+ \|  \tilde{P}^R \dive\omega_h\|_{L^2(\Omega)}^2+ \frac{1}{ \gamma \kappa^2} \|v_h\|_{L^2(\Omega)}^2+\frac{1}{ \gamma \kappa^2} (\Aa \omega_h,\rho_h),
\end{alignat*}
where we used \eqref{311}.
On the other hand, by \eqref{312}
\begin{align*}
    \frac{1}{\gamma \kappa^2} (\Aa \omega_h,\rho_h) &\ge  -\frac{1}{2} \|\omega_h\|_{\Aa}^2- \frac{1}{2 \gamma^2 \kappa^4} \| \rho_h\|_{\Aa}^2 
    \\
    &\ge  -\frac{1}{2} \|  \omega_h\|_{\Aa}^2- \frac{1}{2 \gamma \kappa^2} \| v_h\|_{L^2(\Omega)}^2.
\end{align*}
Hence, 
\begin{alignat*}{1}
\frac{1}{2}\| \omega_h\|_{\Aa}^2+ \|\tilde{P}^R \dive \omega_h\|_{L^2(\Omega)}^2+\frac{1}{2 \gamma \kappa^2}\|v_h\|_{L^2(\Omega)}^2 \le  B(\omega_h, v_h; \tau_h, w_h).
\end{alignat*}
This shows \eqref{thminf1t}. We clearly have \eqref{thminf2t} if we use \eqref{312}.
\end{proof}

Using the above stability result we can prove an a-priori error estimate.
\begin{theorem}
\label{thm:FEMRt-error-estimate}
Let $\sigma, u$ solve \eqref{Elas}  and $\tsigma^R_h, \tu^R_h$ solve \eqref{FEMRt}, then
\begin{alignat*}{1}
& \|\tilde{\Pi}^R \sigma-\tsigma^R_h \|_{\Aa} +  \|\tilde{P}^R \dive ( \tilde{\Pi}^R \sigma-\tsigma^R_h)  \|_{L^2(\Omega)}^2+ \| \tilde{P}^R u-\tu^R_h \|_{L^2(\Omega)} \\
&\le C (\|\tilde{\Pi}^R \sigma-\sigma \|_{\Aa} +
  \| \tilde{P}^R u-u \|_{L^2(\Omega)}).
\end{alignat*}
\end{theorem}
\begin{proof}
By Theorem~\ref{thm:Babuska-Aziz-infsup}, there exist $\tau_h \in \tilde{\Sigma}^R_h$ and $w_h \in \tilde{V}_h^R$, satisfying
\begin{alignat*}{1}
\| \tilde{\Pi}^R \sigma-\tsigma^R_h\|_{\Aa}^2+ \|\tilde{P}^R \dive ( \tilde{\Pi}^R \sigma-\tsigma^R_h)\|_{L^2(\Omega)}^2+\|\tilde{P}^R u-\tu^R_h\|_{L^2(\Omega)}^2 \\
\le c_1  B(\tilde{\Pi}^R \sigma-\tsigma^R_h, \tilde{P}^R u-\tu^R_h; \tau_h, w_h)
\end{alignat*}
and
\begin{alignat}{1} \label{eq:tauw-bound}
 &  \|\tau_h\|_{\Aa}^2+ \|\tilde{P}^R \dive \tau_h\|_{L^2(\Omega)}^2 + \|w_h\|_{L^2(\Omega)}^2  \nonumber \\
 & \le  c_2 (\| \tilde{\Pi}^R \sigma-\tsigma^R_h\|_{\Aa}^2+ \|\tilde{P}^R \dive (\tilde{\Pi}^R \sigma-\tsigma^R_h)\|_{L^2(\Omega)}^2+\|\tilde{P}^R u-\tu^R_h\|_{L^2(\Omega)}^2). 
\end{alignat}
Then, 
\begin{alignat*}{1}
& \|\tilde{\Pi}^R \sigma-\tsigma^R_h \|_{\Aa}^2+ \|\tilde{P}^R \dive ( \tilde{\Pi}^R \sigma-\tsigma^R_h)  \|_{L^2(\Omega)}^2+\| \tilde{P}^R u-\tu^R_h \|_{L^2(\Omega)}^2 \\ 
&\le c_1 \Big( (\Aa (\tilde{\Pi}^R \sigma-\tsigma^R_h, \tau_h)+(\tilde{P}^R u-\tu^R_h , \dive \tau_h) \Big)
\end{alignat*}
where we used \eqref{Elas2}, \eqref{FEM2Rt} and \eqref{commutePiglobalRt} to say the term $(\dive(\tilde{\Pi}^R \sigma-\tsigma^R_h) , w_h)$ vanishes.

On the other hand, using \eqref{Elas1}, \eqref{FEM1Rt} we get
\begin{alignat*}{1}
(\Aa (\tilde{\Pi}^R \sigma-\tsigma^R_h), \tau_h)+(\tilde{P}^R u-\tu^R_h , \dive \tau_h)= (\Aa (\tilde{\Pi}^R \sigma-\sigma), \tau_h)+(\tilde{P}^R u-u , \dive \tau_h).
\end{alignat*}
The result now follows by using \eqref{eq:tauw-bound}.
\end{proof}
As a corollary, we give an error estimate with convergence rate. 
\begin{corollary}
Let $\sigma, u$ solve \eqref{Elas}  and $\tsigma^R_h, \tu^R_h$ solve \eqref{FEMRt}. If  $\sigma  \in H^1(\Omega, \mathbb{S}), u \in [H^1(\Omega)]^3$, then
 \begin{equation}
\|\sigma-\tsigma^R_h\|_{\Aa}+ \|u-\tu^R_h\|_{L^2(\Omega)} \le C  h (\|\sigma\|_{H^1(\Omega)} +\|u\|_{H^1(\Omega)}).
\end{equation}
 \end{corollary}

\begin{remark}
    The argument in the proof of Theorem~\ref{thm:ee-robust} (on robustness near incompressibility) does not extend to this reduced space pair because an analogue of \eqref{divsigR} is not available for $\tsigma_h^R$ and $\tilde{\Pi}^R$. 
\end{remark}

\bibliographystyle{abbrv}

\begin{thebibliography}{10}

\bibitem{Alfel84}
P.~Alfeld.
\newblock A trivariate {Clough}--{Tocher} scheme for tetrahedral data.
\newblock {\em Computer Aided Geometric Design}, 1(2):169--181, 1984.

\bibitem{Amrouche-etal:1998}
C.~Amrouche, C.~Bernardi, M.~Dauge, and V.~Girault.
\newblock Vector potentials in three-dimensional non-smooth domains.
\newblock {\em Math. Methods Appl. Sci.}, 21(9):823--864, 1998.

\bibitem{arnold2008finite}
D.~Arnold, G.~Awanou, and R.~Winther.
\newblock Finite elements for symmetric tensors in three dimensions.
\newblock {\em Mathematics of Computation}, 77(263):1229--1251, 2008.

\bibitem{arnold2007mixed}
D.~Arnold, R.~Falk, and R.~Winther.
\newblock Mixed finite element methods for linear elasticity with weakly
  imposed symmetry.
\newblock {\em Mathematics of Computation}, 76(260):1699--1723, 2007.

\bibitem{arnold2018finite}
D.~N. Arnold.
\newblock {\em Finite element exterior calculus}.
\newblock SIAM, 2018.

\bibitem{arnoldbrezzi1984peers}
D.~N. Arnold, F.~Brezzi, and J.~Douglas, Jr.
\newblock P{EERS}: a new mixed finite element for plane elasticity.
\newblock {\em Japan J. Appl. Math.}, 1(2):347--367, 1984.

\bibitem{arnold1984family}
D.~N. Arnold, J.~Douglas, and C.~P. Gupta.
\newblock A family of higher order mixed finite element methods for plane
  elasticity.
\newblock {\em Numerische Mathematik}, 45:1--22, 1984.

\bibitem{arnold2006finite}
D.~N. Arnold, R.~S. Falk, and R.~Winther.
\newblock Finite element exterior calculus, homological techniques, and
  applications.
\newblock {\em Acta Numerica}, 15:1--155, 2006.

\bibitem{arnold2021complexes}
D.~N. Arnold and K.~Hu.
\newblock Complexes from complexes.
\newblock {\em Found. Comput. Math.}, 21(6):1739--1774, 2021.

\bibitem{ArnolQin92}
D.~N. Arnold and J.~Qin.
\newblock Quadratic velocity/linear pressure {S}tokes elements.
\newblock In R.~Vichnevetsky, D.~Knight, and G.~Richter, editors, {\em
  {Advances in Computer Methods for Partial Differential Equations-VII}}, pages
  28--34. IMACS, 1992.

\bibitem{arnold2002mixed}
D.~N. Arnold and R.~Winther.
\newblock Mixed finite elements for elasticity.
\newblock {\em Numerische Mathematik}, 92(3):401--419, 2002.

\bibitem{Bebendorf-poincare}
M.~Bebendorf.
\newblock A note on the {P}oincar\'{e} inequality for convex domains.
\newblock {\em Z. Anal. Anwendungen}, 22(4):751--756, 2003.

\bibitem{boffietal2009elasticity}
D.~Boffi, F.~Brezzi, and M.~Fortin.
\newblock Reduced symmetry elements in linear elasticity.
\newblock {\em Commun. Pure Appl. Anal.}, 8(1):95--121, 2009.

\bibitem{BoffiBrezziFortin-2013}
D.~Boffi, F.~Brezzi, and M.~Fortin.
\newblock {\em Mixed finite element methods and applications}, volume~44 of
  {\em Springer Series in Computational Mathematics}.
\newblock Springer, Heidelberg, 2013.

\bibitem{Brezzi-Fortin:1991}
F.~Brezzi and M.~Fortin.
\newblock {\em Mixed and hybrid finite element methods}, volume~15 of {\em
  Springer Series in Computational Mathematics}.
\newblock Springer-Verlag, New York, 1991.

\bibitem{Brubeck-Kirby:2025}
P.~D. Brubeck and R.~C. Kirby.
\newblock {F}{I}{A}{T}: enabling classical and modern macroelements.
\newblock {\em arXiv preprint arXiv:2501.14599}, 2025.

\bibitem{calabi1961compact}
E.~Calabi.
\newblock {On compact, Riemannian manifolds with constant curvature I}.
\newblock {\em Proc. Sympos. Pure Math., Amer. Math. Soc., Providence, RI,
  1961}, pages 155--180, 1961.

\bibitem{christiansen2020discrete}
S.~Christiansen, J.~Gopalakrishnan, J.~Guzm{\'{a}}n, and K.~Hu.
\newblock A discrete elasticity complex on three-dimensional {Alfeld} splits.
\newblock {\em Numer. Math.}, 156:159--204, 2024.

\bibitem{Christiansen-Hu:2018}
S.~H. Christiansen and K.~Hu.
\newblock Generalized finite element systems for smooth differential forms and
  {S}tokes' problem.
\newblock {\em Numer. Math.}, 140(2):327--371, 2018.

\bibitem{christiansen2019finite}
S.~H. Christiansen and K.~Hu.
\newblock Finite element systems for vector bundles: Elasticity and curvature.
\newblock {\em Foundations of Computational Mathematics}, 23(2):545--596, 2023.

\bibitem{Ciarlet:elasticity-1}
P.~G. Ciarlet.
\newblock {\em Mathematical elasticity. {V}ol. {I}}, volume~20 of {\em Studies
  in Mathematics and its Applications}.
\newblock North-Holland Publishing Co., Amsterdam, 1988.
\newblock Three-dimensional elasticity.

\bibitem{Ciarl91}
P.~G. Ciarlet.
\newblock Basic error estimates for elliptic problems.
\newblock In {\em Handbook of Numerical Analysis, {V}ol.\ {II}}, Handb. Numer.
  Anal., II, pages 17--351. North-Holland, Amsterdam, 1991.

\bibitem{ciarlet2009intrinsic}
P.~G. Ciarlet, L.~Gratie, and C.~Mardare.
\newblock Intrinsic methods in elasticity: a mathematical survey.
\newblock {\em Discrete \& Continuous Dynamical Systems-A}, 23(1\&2):133, 2009.

\bibitem{clough1965finite}
R.~W. Clough.
\newblock Finite element stiffness matricess for analysis of plate bending.
\newblock In {\em Proc. of the First Conf. on Matrix Methods in Struct. Mech.},
  pages 515--546, 1965.

\bibitem{costabel2010bogovskiui}
M.~Costabel and A.~McIntosh.
\newblock {On Bogovski{\u\i} and regularized Poincar{\'e} integral operators
  for de Rham complexes on Lipschitz domains}.
\newblock {\em Mathematische Zeitschrift}, 265(2):297--320, 2010.

\bibitem{eastwood2000complex}
M.~Eastwood.
\newblock A complex from linear elasticity.
\newblock In {\em Proceedings of the 19th Winter School" Geometry and
  Physics"}, pages 23--29. Circolo Matematico di Palermo, 2000.

\bibitem{Evans-book}
L.~C. Evans.
\newblock {\em Partial differential equations}, volume~19 of {\em Graduate
  Studies in Mathematics}.
\newblock American Mathematical Society, Providence, RI, 1998.

\bibitem{Falk08}
R.~S. Falk.
\newblock Finite element methods for linear elasticity.
\newblock In {\em Mixed Finite Elements, Compatibility Conditions, and
  Applications}, volume 1939 of {\em Lecture Notes in Mathematics}, pages
  160--194. Springer Verlag, 2008.

\bibitem{farhloul1997elasticity}
M.~Farhloul and M.~Fortin.
\newblock Dual hybrid methods for the elasticity and the {S}tokes problems: a
  unified approach.
\newblock {\em Numer. Math.}, 76(4):419--440, 1997.

\bibitem{fu2018exact}
G.~Fu, J.~Guzman, and M.~Neilan.
\newblock {Exact smooth piecewise polynomial sequences on Alfeld splits}.
\newblock {\em Mathematics of Computation}, 2020.

\bibitem{gong2023discrete}
S.~Gong, J.~Gopalakrishnan, J.~Guzm\'{a}n, and M.~Neilan.
\newblock Discrete elasticity exact sequences on {W}orsey-{F}arin splits.
\newblock {\em ESAIM Math. Model. Numer. Anal.}, 57(6):3373--3402, 2023.

\bibitem{gopalakrishnan2012second}
J.~Gopalakrishnan and J.~Guzm{\'a}n.
\newblock A second elasticity element using the matrix bubble.
\newblock {\em IMA Journal of Numerical Analysis}, 32(1):352--372, 2012.

\bibitem{guzman2018inf}
J.~Guzm{\'a}n and M.~Neilan.
\newblock Inf-sup stable finite elements on barycentric refinements producing
  divergence--free approximations in arbitrary dimensions.
\newblock {\em SIAM Journal on Numerical Analysis}, 56(5):2826--2844, 2018.

\bibitem{hu2015family}
J.~Hu and S.~Zhang.
\newblock A family of symmetric mixed finite elements for linear elasticity on
  tetrahedral grids.
\newblock {\em Sci. China Math.}, 58(2):297--307, 2015.

\bibitem{hu2016low}
J.~Hu and S.~Zhang.
\newblock Finite element approximations of symmetric tensors on simplicial
  grids in {$\Bbb{R}^n$}: the lower order case.
\newblock {\em Math. Models Methods Appl. Sci.}, 26(9):1649--1669, 2016.

\bibitem{Huang-2023lowordermixed}
X.~Huang, C.~Zhang, Y.~Zhou, and Y.~Zhu.
\newblock New low-order mixed finite element methods for linear elasticity.
\newblock {\em Adv. Comput. Math.}, 50(2):Paper No. 17, 2024.

\bibitem{johnson1978some}
C.~Johnson and B.~Mercier.
\newblock Some equilibrium finite element methods for two-dimensional
  elasticity problems.
\newblock {\em Numerische Mathematik}, 30(1):103--116, 1978.

\bibitem{Krizek1982}
M.~K\v{r}\'{\i}\v{z}ek.
\newblock An equilibrium finite element method in three-dimensional elasticity.
\newblock {\em Apl. Mat.}, 27(1):46--75, 1982.
\newblock With a loose Russian summary.

\bibitem{Stenberg-2023}
P.~L. Lederer and R.~Stenberg.
\newblock Energy norm analysis of exactly symmetric mixed finite elements for
  linear elasticity.
\newblock {\em Math. Comp.}, 92(340):583--605, 2023.

\bibitem{lee2016weaksymmetry}
J.~J. Lee.
\newblock Towards a unified analysis of mixed methods for elasticity with
  weakly symmetric stress.
\newblock {\em Adv. Comput. Math.}, 42(2):361--376, 2016.

\bibitem{MarsdHughe94}
J.~E. Marsden and T.~J.~R. Hughes.
\newblock {\em Mathematical foundations of elasticity}.
\newblock Dover Publications Inc., New York, 1994.

\bibitem{morley1989elasticity}
M.~E. Morley.
\newblock A family of mixed finite elements for linear elasticity.
\newblock {\em Numer. Math.}, 55(6):633--666, 1989.

\bibitem{PitkaStenb83}
J.~Pitk{\"a}ranta and R.~Stenberg.
\newblock Analysis of some mixed finite element methods for plane elasticity
  equations.
\newblock {\em Math. Comp}, 41(164):399--423, 1983.

\bibitem{2025-defelement}
M.~W. Scroggs, P.~D. Brubeck, J.~P. Dean, J.~S. Dokken, and I.~Marsden.
\newblock {DefElement:} an encyclopedia of finite element definitions.
\newblock submitted to Computational Science and Engineering, 2025.

\bibitem{stenberg1988elasticity}
R.~Stenberg.
\newblock A family of mixed finite elements for the elasticity problem.
\newblock {\em Numer. Math.}, 53(5):513--538, 1988.

\bibitem{stenberg1991postprocessing}
R.~Stenberg.
\newblock Postprocessing schemes for some mixed finite elements.
\newblock {\em RAIRO Mod\'{e}l. Math. Anal. Num\'{e}r.}, 25(1):151--167, 1991.

\bibitem{Watwood1968AnES}
V.~B. Watwood and B.~J. Hartz.
\newblock An equilibrium stress field model for finite element solutions of
  two-dimensional elastostatic problems.
\newblock {\em International Journal of Solids and Structures}, 4:857--873,
  1968.

\bibitem{Zhang-barycentric:2005}
S.~Zhang.
\newblock A new family of stable mixed finite elements for the 3{D} {S}tokes
  equations.
\newblock {\em Math. Comp.}, 74(250):543--554, 2005.

\bibitem{zhang2005stokes}
S.~Zhang.
\newblock A new family of stable mixed finite elements for the 3{D} {S}tokes
  equations.
\newblock {\em Math. Comp.}, 74(250):543--554, 2005.

\end{thebibliography}

\end{document}